\documentclass{article} 
\usepackage[top=2cm, bottom=2cm, left=3cm, right=3cm]{geometry}

\usepackage{authblk}
\usepackage{cite}

\usepackage[T1]{fontenc}
\usepackage{braket}
\usepackage{float}
\usepackage{amsmath}
\usepackage{amsfonts}
\usepackage{amsthm}
\usepackage{amssymb}
\usepackage{graphicx}
\usepackage{xcolor}
\usepackage{verbatim}
\usepackage{mathtools}
\usepackage{epstopdf}
\usepackage{appendix}

\numberwithin{equation}{section}

\newcommand{\bgeq}{\begin{equation}}
\newcommand{\eeq}{\end{equation}}

\newcommand{\bgc}{\begin{cases}}
\newcommand{\ec}{\end{cases}}
\newcommand{\bgmx}{\begin{bmatrix}}
\newcommand{\emx}{\end{bmatrix}}
\newcommand{\e}{\ensuremath{\mathrm{e}}}
\newcommand{\bd}[1]{\ensuremath{\boldsymbol{#1}}}

\newcommand{\I}{\ensuremath{\mathrm{i}}}

\newcommand{\lt}{\mathopen{}\mathclose\bgroup\left}
\newcommand{\rt}{\aftergroup\egroup\right}

\newcommand{\dd}{\ensuremath{\mathrm{d}}}

\newcommand{\const}{\ensuremath{\text{const.}}}

\newcommand{\deq}{\ensuremath{\overset{d}{=}}}
\DeclareMathOperator*{\pr}{P}

\DeclareMathOperator*{\E}{\mathbb{E}}
\newcommand{\EE}[1]{\E\lt[ #1 \rt]}
\DeclareMathOperator*{\var}{{var}}
\DeclareMathOperator*{\cov}{{cov}}

\newcommand{\XXX}[1]{{\color{red} #1}}

\definecolor{darkgreen}{rgb}{0.0, 0.5, 0.0}

\newtheorem{prp}{Proposition}
\theoremstyle{definition}
\newtheorem{ex}{Example}
\newtheorem{dfn}{Definition}
\newtheorem{rem}{Remark} 

\usepackage{hyperref}
\hypersetup{
    colorlinks,
    linkcolor={red!50!black},
    citecolor={blue!50!black},
    urlcolor={blue!80!black}
}

\newcommand\indep{\protect\mathpalette{\protect\independenT}{\perp}}
\def\independenT#1#2{\mathrel{\rlap{$#1#2$}\mkern2mu{#1#2}}}

\author{Jakub \'Sl\k{e}zak}
\date{15 December 2025}
\title{Codifference as a measure of dispersion and dependence for mixture models}

\begin{document}
\maketitle

\begin{abstract}
    Codifference is a commonly used measure of dependence for stable vectors and processes for which covariance is infinite. However, we argue that it can also be used for other heavy-tail distributions and it provides useful information for other non-Gaussian distributions as well, no matter the tails. Motivated by this, we analyse codifference using as little assumptions as possible about the studied model. It leads us to propose its natural domain and three natural variants of it. Using the wide class of variable scale mixture distributions we argue that the codifference can be interpreted as the measure of bulk properties which ignores the tails much more than the covariance. It can also detect forms of non-linear memory which covariance cannot. Finally, we show the asymptotic distribution of its estimator.
\end{abstract}

\section{Introduction to dependence measures}

The correlation and covariance are so universally used as the measures of random variables dependence that among non-experts they are commonly mistaken to be synonymous. In reality the dependence may be vastly more complex than the value of the correlation alone (in fact arbitrarily so), but in practical circumstances the choice of quantities to describe it which can be analytically analysed and efficiently estimated is often quite limited. The correlation/covariance is indeed a very reasonable choice in most of the circumstances, having its analytical properties studied for nearly all widely used models and having a simple, well-behaved estimator. Moreover, it has strong and intuitive geometrical interpretation, as it is the projection of one data vector onto the another. More abstractly speaking, random variables with finite second order form a Hilbert space in which scalar product is the covariance. In this view, covariance tells everything that can even be told about linear dependence of random variables.

That being said, the lack of commonly used alternatives may be considered surprising, given the crucial importance of the dependence itself. If one would like to list the quantities which were analytically studied for at least large portion of commonly used models, one is left with mostly three choices: moments, probability density itself (which is often problematic to estimate), or characteristic function. The latter has an easy estimator with a very good behaviour (as it is bounded) and is related to the density through the Fourier transform, a technique known among most of the scientific community. Thus, it makes a natural basis for many useful techniques.

The idea of analysing dependence through the characteristic function was introduced very early in the development of probability theory, but the tools which we will discuss originate in the theory of stable variables \cite{taqqu,janicki}. These distributions -- when non-Gaussian -- do not have a finite second moment, and most of their densities are analytically unwieldy.

The idea behind codifference for stable variables is intuitive. For Gaussian variables we have a natural measure of the distributions' scale - the standard deviation $\sigma$. Dependence between two variables $X$ and $Y$ can be then defined as difference between the (squared) scale $\sigma^2(X-Y)$ and scales $\sigma^2(X) + \sigma^2(Y)$. The result is nothing but the covariance. Similarly, stable variables also have a natural measure of the scale - the constant $c$ which appears in the formula for their  characteristic function $\EE{\exp(\I\theta X)} = \exp(-c^\alpha|\theta|^\alpha)$; here we assume that the stable variable is symmetric. And again, a dependence measure can be constructed by quantifying the difference between $c^\alpha(X-Y)$ and $c^\alpha(X)+c^\alpha(Y)$. This is the classical definition of codifference. As a side note, its name seems to be mostly a historical residue related to other measures of stable variables, and it doesn't have much inherent sense. We keep it to be consistent with the already existing literature.

In this context numerous works use measures of variation, covariation \cite{miller78, weron84, zak15}, codifference \cite{nowicka97-1,nowicka97-2, gallagher01, chechkinCod, kruczek, crossCod} and related quantities \cite{podgorski91, shao93, zak19} to study dependence in theoretical or practical  setting. 

Concurrently, a similar notion of dynamical functional was discussed in the context of dependence and ergodic properties of L{\'e}vy processes \cite{rosinski97,magdziarz11, loch16,loch18,loch19}. It was also noted that it can detect non-linear ergodicity breaking and non-Gaussianity for some non-stable variables \cite{slezak2019}.

We argue that these concepts are not strictly limited to their initial scope and found some distinct applications and should be more widely known.  As we will show they provide additional information for non-Gaussian distributions and they might be relevant for models in which this property is important.

The topic of dependence for such models was researched using multiple approaches, the most notable being copula theory \cite{copulaBook}. However, we will show that for the very wide important class of mixture models the codifference is a practical tool which could be a valuable complement to the already used methods, quantifying non-Gaussianity and non-linearity for these models.

Functions for estimating the considered functions and their asymptotic confidence intervals are available at \url{https://github.com/jaksle/Codifference.jl}.

\section{Overview of the codifference}
In this section, we recall the definition of the classical codifference \cite{codi1,codi2,nowicka97-1,nowicka97-2,crossCod,Grzesiek2020}, then propose and discuss its natural variants. Afterwards, we present the properties of the discussed measures and show how they behave for few exemplary models.

Let us start with setting the regularity conditions we require to use the codifference. They guarantee that the quantities we describe have a reasonable practical interpretation.

\begin{dfn}[Basic domain of the codifference]\label{dfn:domain} We denote by $\mathcal D$ the class of one- or multi-dimensional variables with strictly positive-definite (Lebesgue) density function together with a limiting case of atom at 0.
\end{dfn}
The positive-definiteness implies that the characteristic function $\varphi_X(\theta) \coloneqq \EE{\e^{\I\theta X}} $  is real-valued and positive for any $X\in\mathcal{D}$ \footnote{Bochner theorem guarantees equivalence of these conditions for a large class of functions.}. Moreover, the Riemann-Lebesgue lemma implies that $\varphi_X(\theta)$ is decaying to zero as $\theta\to \pm\infty$. Note that these two features characterise also some variables with singular distribution (and thus without Lebesgue density), however they seem to be of little use in practical statistical analysis so we will not consider them further on. We also allow for variables with a null component $X=0$, which corresponds to constant characteristic function $\varphi_X(\theta) = 1$. It can appear in expressions like $X-X$ and will result in zero scale or dependence.

Being positive-definite, the probability density functions for any random variable in $\mathcal D$ in particular is symmetric. Therefore, the mean $\EE{X}$ and all odd moments vanish. One can rather easily extend the domain $\mathcal D$ by centring all variables in the definitions below (which for many types of distributions will then have positive-definite density, e.g. for Gaussian ones), but it will only complicate formulas, so we will skip this straightforward generalisation.

The domain in Def. \ref{dfn:domain} can be refined by postulating that $\varphi_X$ decays to zero monotonically, i.e. is unimodal. This requirement does not affect the discussed properties of the codifference and related functions, but gives it another layer of regular behaviour and is true for many often considered models anyway. This will become clear after we discuss the basic properties. Additionally in Appendix \ref{app:unimodal} we describe more general conditions suitable for multidimensional setting.

\subsection{Measuring dispersion}

We start with a simple function which is useful by itself and will be used to define memory measures further on.
\begin{dfn} For $X\in \mathcal D$ we define a dispersion measure \label{dfn:lcf}
 \begin{equation}
    l^\theta(X) \coloneqq -\frac{2}{\theta^2}\ln\E\big[\e^{\I\theta X}\big]
\end{equation}
and henceforth we will call it the logarithm of characteristic function (lcf in short). Note that $l^\theta(X) = l^1(\theta X)$, so often case $\theta = 1$ is sufficient, which we will take as default and write $l\coloneqq l^1$.
\end{dfn}

\begin{rem}[The lcf and cumulant generating function]
The lcf defined above is nothing but the cumulant generating function taken at imaginary argument and rescaled by its second Taylor coefficient. This rescaling is for the lcf to have a unit of variance and collapse to it for Gaussian variables. For cumulant genrating function some authors prefer the imaginary argument over the real one (that is over $\ln \EE{e^{t X}}$) because it is well-defined for all random variables, not only those with sub-exponential tails. In this case the name \emph{second characteristic} of the distribution is sometimes used in the literature \cite{lukacs}, however we prefer the more descriptive name instead. The major difference between the lcf and moment generating function is how it will be used. Moment generating function is considered mainly as a function of $\theta$ which contains information about one given distribution. We will fix the parameter $\theta$ and then compare its values for many distributions. Of course additional information provided by multiple $\theta$ still can be used, but we will mostly concentrate on other aspects.

Because of this relation the lcf can be expressed using the cumulants $\kappa_n=\kappa_n(X)$ \cite{cumulants}
\begin{equation}\label{eq:cumulants}
l^\theta(X) = 2\sum_{k=0}^\infty (-1)^k\frac{\kappa_{2k+2}(X)}{(2k+2)!}\theta^{2k}, 
\end{equation}
where the sum should be cut at appropriate term if some higher-order moments of $X$ do not exist. The first two terms of this expression are variance $\sigma^2$ corrected by the excess kurtosis $K$, $l^\theta(X) = \sigma^2-K\sigma^4\theta^2/12+\mathcal O(\theta^4)$. This means that for distributions with high kurtosis (\emph{leptokurtic}) we may expect the lcf to be smaller than variance, at least for small enough $\theta$. This suggests that the lcf in a sense understresses the tails of the distribution, which on the other hand strongly affect kurtosis\footnote{Note however that, if finite, we have arbitrary kurtosis with arbitrary tails, so kurtosis is not a good measure of the ditribution's tails. Its use is informally stated as "measuring how wide the distribution's waist is".}. We can make this intuition more formal by the following proposition.

\end{rem}

\begin{prp}\label{prp:bulk}
    Let us decompose a random variable $X\in\mathcal D$ into a bulk and tails by writing
    \begin{equation}
        X = \begin{cases} 
            X_{\mathrm{bulk}}& \textnormal{ with probability } p,\\
            w X_{\mathrm{tail}}& \textnormal{ with probability } 1-p,
        \end{cases}
    \end{equation}
    $\pr(X_\textnormal{tail} =0)=0$. Then in the limit of far away tails
    \begin{equation}
        l^\theta(X) \xrightarrow{w\to\infty}l^\theta(X_\mathrm{bulk}) -\frac{2\ln p}{\theta^2}.
    \end{equation}
\end{prp}
\begin{proof}
    This is a Riemann-Lebesgue lemma in disguise, $X_\mathrm{tail}$ having a density implies $\EE{\e^{\I\theta w X_\mathrm{tail}}}\to 0$ as $w\to\infty$. But
    \begin{equation}
        l^\theta(X)  = -\frac{2}{\theta^2}\ln\left(p\EE{\e^{\I\theta X_\mathrm{bulk }}} + (1-p) \EE{\e^{\I\theta w X_\mathrm{tail}}}\right).
    \end{equation}
\end{proof}
\noindent Compare this result to the covariance, which would diverge to infinity. It is also worth to notice that the positive term $-2\theta^{-2}\ln p$ accounts for the probability mass 'lost' in the tails and guarantees that the lcf correctly reacts to larger spread as it increases with $p\to 0$. 

The statement of this result is illustrated in Fig. \ref{fig:LCFlim}.
\begin{figure}\centering

\includegraphics[width=14cm]{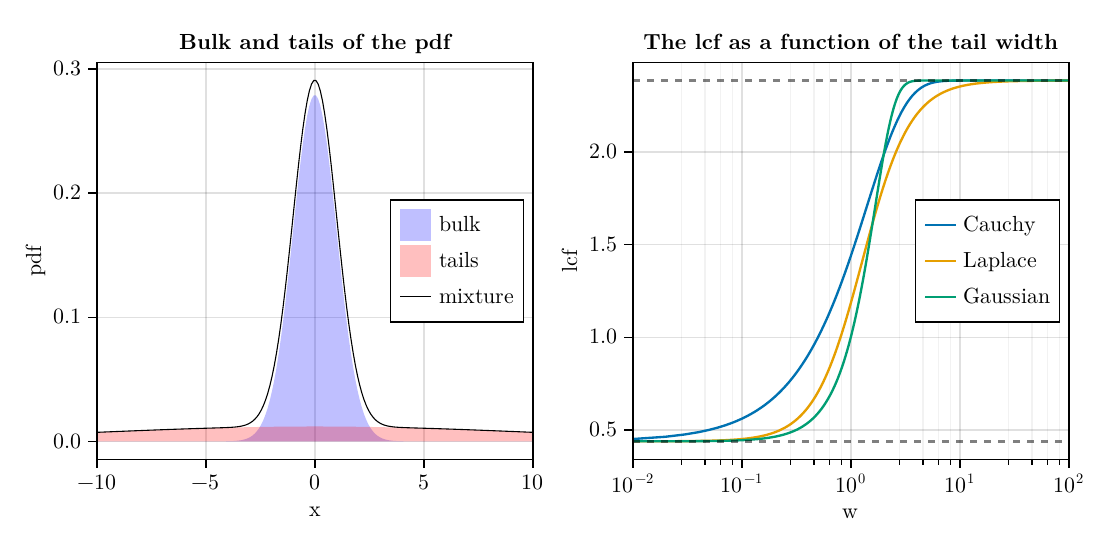}
\caption{Illustration of Prop. \ref{prp:bulk}. Left: an exemplary mixture with medium scale bulk component and large scale tail component, here both are Gaussian. Right: the lcf has universal limits for both $w\to0^+$ and $w\to\infty$ which do not depend on the tail component distribution. The limit $w\to0^+$ can be interpreted as role switching: the bulk becomes the tails, and the tails collapse to $X=0$ variable.}\label{fig:LCFlim}
\end{figure}

The most essential features of the lcf follow directly from the elementary properties of the characteristic function.
\begin{prp}
Basic properties of the lcf: \label{prp:lcf}
\begin{itemize}
\item[a)] It is non-negative and equals 0 only for $X=0$.
\item[b)] It increases as the scale of the variable increases
\begin{equation}
l^\theta(wX)\xrightarrow{w\to\infty}\infty
\end{equation}
with the exception of atom at 0 present, for which $l^\theta(wX)\to -2\theta^{-2}\ln\pr(X=0),w\to\infty$.
\item[c)]It decouples for sums of independent variables, for $X_1\indep X_2$
\begin{equation}
l^\theta(X_1+X_2) = l^\theta(X_1)+l^\theta(X_2).
\end{equation}
\item[d)] For symmetric Gaussian variables it is equal to variance.
\item[e)] For variables with finite second moment it converges to variance as $\theta\to 0$.
\end{itemize}
\end{prp}

\begin{rem}[The lcf as a quasi-mean]\label{rem:qm} The lcf, standard deviation, and other $p$-norms of random variables can be all viewed as particular cases of the more general concept of \emph{quasi-mean}, sometimes also called \emph{Kolmogorov expected value} \cite{kolmogorov,inequalities}. Quasi-mean of random variable based on function $f$ is given by
\begin{equation}
\mathcal E_f[X]\coloneqq f^{-1}(\EE{f(X)}).
\end{equation}
So, for variables with zero mean standard deviation is square-mean, in general $p$-norm is $p$th power-mean, inverse-mean is harmonic mean and lcf is, up to a factor cis-mean ('cis' being common name for $\exp(\I(\boldsymbol{\cdot})))$. Most of the theoretical research about quasi-mean is concentrated on axiomatic approach, identifying regularity assumptions which single out quasi-means from other functionals, especially for its sample equivalent $f^{-1}(n^{-1}\sum_{k=1}^nf(x_k))$.

This line of work does not seem to have direct consequences for applications, but it may be used to offer another interpretation of the lcf. We can think of $f$ as test function which fist applies weights to the probability mass and then is used to construct a specific functional which collapses for deterministic variables, $\mathcal E_f[\mu] = \mu$. Therefore $\mathcal E_f[X-\EE{X}]$ is zero for deterministic $X=\mu=\EE{X}$ and consequently deviation from zero can be used as a measure of being far from deterministic, quite close to the intuitive notion of the distribution's spread.

Among different functions $f$, linear mean $f(x) = a x$ and exp-mean $f(x)=\exp(\beta x)$ ensure that $\mathcal E_f[X_1+X_2]=\mathcal E_f[X_1]+\mathcal E_f[X_1]$ for independent variables. Taking $\beta$ real is also possible, but limits the applications to distributions with sufficiently short tails.
\end{rem}

\subsection{Measuring dependence}

\begin{dfn}\label{def:main}
For pairs of variables fulfilling $[X,Y]\in \mathcal D$ we define few measures of memory using lcf as a base.
\begin{itemize} 
    \item[a)] Symmetric codifference (scdf)\label{dfn:scdf}
    \begin{equation}
    s^\theta (X,Y)\coloneqq \frac{1}{4}\left(l^\theta(X+Y)-l^\theta(X-Y)\right)=\frac{1}{2\theta^2}\ln\frac{\EE{\e^{\I\theta(X-Y)}}}{\EE{\e^{\I\theta(X+Y)}}}  
    \end{equation}
    \item[b)] Asymmetric (classical) codifference (acdf) \cite{crossCod}
    \begin{equation}\label{dfn:acdf}
    c_\pm^\theta (X,Y)\coloneqq \pm\frac{1}{2}\left(l^\theta(X)+l^\theta(Y)-l^\theta(X\mp Y)\right) =\pm\frac{1}{\theta^2}\ln\frac{\E\big[\e^{\I\theta(X\mp Y)}\big]}{\EE{\e^{\I\theta X}}\EE{\e^{\I\theta Y}}}
    \end{equation}

\end{itemize}
Moreover we mention a third related function which appears in the literature, the dynamical functional \cite{podgorski91,magdziarz11, loch16,loch18,loch19}
    \begin{equation}\label{dfn:df}
    d^\theta(X,Y)\coloneqq \e^{-\theta^2l^\theta(X-Y)/2}-\e^{-\theta^2(l^\theta(X)+l^\theta(Y))/2} =  \E\left[\e^{\I\theta(X-Y)}\right] - \EE{\e^{\I\theta  X}}\EE{\e^{\I\theta Y}}
    \end{equation}

In the same manner as before the value $\theta=1$ will not be written and therefore we have $c_\pm(X,Y) = c_\pm^1(X,Y)$ etc.
\end{dfn}
In Appendix \ref{app:linSpace} we also introduce and analyse a more general definition suitable for using codifference in higher-dimensional spaces, though it not required for understanding the rest of the paper.

\begin{rem}[Classical definition] The codifference most commonly considered in the stable distributions literature in our notation is $c_+$.
\end{rem}

\begin{rem}[Polarisation identities and symmetry]
The definitions a) and b) are  nothing but polarisation identities. For any scalar product $\langle\bd\cdot,\bd\cdot\rangle$ with the induced norm $\Vert\bd\cdot\Vert^2$
\begin{align}
    \langle x,y\rangle = \frac{1}{4}\left(\Vert x+y\Vert^2-\Vert x-y\Vert^2\right) 
    = \frac{1}{2}\left(\Vert x\Vert^2+\Vert y\Vert^2-\Vert x-y\Vert^2\right) 
    = \frac{1}{2}\left(\Vert x+y\Vert^2 -\Vert x\Vert^2-\Vert y\Vert^2\right).
\end{align}
In fact if the two right equations hold, then each can be used to define a scalar product by the left equation, and scalar product can be interpreted as a similarity measure. The lcf is not  a norm, but as we demonstrate these identities impose enough regularity that a reasonable measure of similarity is obtained.

Notably, the first identity $\frac{1}{4}\left(\Vert x+y\Vert^2-\Vert x-y\Vert^2\right)$ imposes strong symmetry. Let us take some variables $A, B$ and take $X,Y$ as those variables rotated by 45$^\circ$ up to a scale, $X\coloneqq A+B,Y\coloneqq A-B$. Then any dependence measure which fits it is a function of the scale on $A$ and scale on $B$ only, in particular it ignores dependence between them. For no dependence $X$ and $Y$ have the same scale, for weak dependence they vary in scale, one going to 0 in the limit of $A\pm B$. Both covariance and symmetric codifference do not change their value in any case.
\end{rem}

\begin{prp}[Relations between codifferences and lcf] The lcf can be expressed using acdf or scdf and scdf can be expressed in terms of acdf,
 \begin{align}
     l^\theta(X) &= c_+^\theta(X,X) = -c_-^\theta(X,-X) = 4s^\theta(X/2,X/2)\label{eq:lcfRel},\\
     s^\theta(X,Y) &= \frac{1}{2}\big(c_+^\theta(X,Y)+c_-^\theta(X,-Y)\big).
 \end{align}
 To express acdf using scdf one may substitute \eqref{eq:lcfRel} into \eqref{dfn:acdf}.
\end{prp}

\begin{rem}[Scale and units]  If the variables $X, Y$ have unit $[j]$ then $\theta$ must have unit $[j^{-1}]$ (as $\e^{\I(\boldsymbol{\cdot})}$ is a transcendental function) and the lcf and codifference have unit $[j^2]$, the same as variance and covariance. This mean they can be compared and plotted on the same axis. If $X$ and $Y$ have different units or scale, Definition \ref{def:main} cannot be used. A natural generalisation is then
\begin{equation}
s^{\theta_1,\theta_2} (X,Y) \coloneqq \frac{1}{2\theta_1\theta_2} \ln\frac{\EE{\e^{\I(\theta_1 X-\theta_2Y)}}}{\EE{\e^{\I(\theta_1X+\theta_2Y)}}},
\end{equation}
similarly for other definitions. It is the transformed form of full two dimensional characteristic function of the vector $[X,Y]$ so it is close to containing all information about its distribution. This larger generality and information content comes at the cost of practicality as analysing two parameter functional is more complicated and it is harder to interpret. Most of our considerations apply to this general form as well, with only minor technical changes, but we will not discuss it for the sake of simplicity and clarity. Often considering even more reduced $\theta=1$ case is sufficient as it can already reveal crucial aspects of the dependence.

The unit $[j^2]$ may be an important feature for many practical applications, but in some cases it is not so important, for example when one considers models with no second moment. In such cases the prefactor $1/\theta^2$ in the definition may be skipped.
\end{rem}

\begin{prp}
Basic properties of the codifference and dynamical functional:\label{prp:basicProp}
\begin{itemize}
    \item[a)] They do not depend on the argument order,
    $s^\theta(X,Y)=s^\theta(Y,X)$, $c_\pm^\theta(X,Y)=c_\pm^\theta(Y,X)$, 
    $d^\theta(X,Y)=d^\theta(Y,X)$.
    \item[b)] Symmetric codifference is symmetric under reflection, $s^\theta(X,-Y)=-s^\theta(X,Y)$.
    \item[c)] For independent variables $X\indep Y$,  $s^\theta(X,Y) = c_\pm^\theta(X,Y)= d^\theta(X,Y)=0$.
    \item[d)] For independent vectors $[X_1,Y_1] \indep [X_2,Y_2]$ the codifference decouples 
    \begin{gather*}
    s^\theta(X_1+X_2,Y_1+Y_2) = s^\theta(X_1,Y_1)+s^\theta(X_2,Y_2),\\
      c_\pm^\theta(X_1+X_2,Y_1+Y_2) = c_\pm^\theta(X_1,Y_1)+c_\pm^\theta(X_2,Y_2). 
    \end{gather*}
    \item[e)] The codifference detects strong dependence, for the case $X=Y$ we have $s^\theta(X,X), c_+^\theta(X,X), d^\theta(X,X)>0$ and for the case $X=-Y$ we have $s^\theta(X,Y), c_-^\theta(X,Y)<0$.
    \item[f)] For Gaussian variables codifference equals covariance.
    \item[g)] For variables with finite second moment codifference converges to covariance as $\theta\to 0$.
    \item[h)] For variables with unimodal $\varphi_X$ symmetric codifference preserves the sign of linear dependence, i.e. taking $Y = \lambda X$ scdf $s^\theta(X,\lambda X)$ is positive for positive $\lambda$ and negative for negative $\lambda$.
    \item[i)] The sign of the symmetric codifference is determined by $\EE{\sin(\theta X)\sin(\theta Y)}$.
\end{itemize}
\end{prp}

\begin{proof} Points a) and b) stem from the fact that for symmetric variables $X-Y\deq -(X-Y) = Y-X$. Points c) and d) are straightforward consequences of the additivity of the lcf. Point e) is equivalent to the property $\varphi_X(\theta)\le 1$. 

Points f) and g) are natural extensions of the lcf properties, given that codifference variants are defined by polarisation identities.

In point h) the sign of scdf depends on the ratio of $\varphi_X(1-\lambda)$ and $\varphi_X(1+\lambda)$ which in turn for unimodal $\varphi_X$ depends on the ratio of $|1-\lambda|$ and $|1+\lambda|$. The first one is smaller for $\lambda>0$, vice versa for $\lambda < 0$.

Finally, for i), using the equality $\cos(X-Y)-\cos(X+Y) =2 \sin(X)\sin(Y)$ we express the symmetric codifference as
\begin{equation}
s(X,Y) = \frac{1}{2}\ln\left(1+2\frac{\EE{\sin(X)\sin(Y)}}{\EE{\cos(X+Y)}}\right).
\end{equation}
The sign of $\ln(1+x)$ is determined by $x$ and $\EE{\cos(X+Y)} = \varphi_{X+Y}(1)$ is positive.
\end{proof}

We can think of properties a)-e) are basic regularity features which guarantee that the codifference is a numerical measure which agrees with the intuitive notion of 'dependence'. They are also important for various practical reasons, e.g. because of a) we do not need to be concerned with the order of the considered variables, d) shows how to easily include the effect of the noise to the statistical analysis. Properties f) and g) establish a link between the codifference and covariance

What is considered strong/weak or positive/negative 'dependence' depends on the context, but most would agree that the distribution highly concentrated around the line $X=Y$ should be considered strong positive dependence and the one concentrated around $X=-Y$ a strong negative one. Point e) assures that the former is always true, but the latter is guaranteed only for symmetric codifference, which also reasonably quantifies the concentration around $Y=\lambda X$ (point h)). The sign of acdf $c_+^\theta$ and dynamical functional for $Y=-X$ is determined by whether the ratio of $\varphi_X(2\theta)/\varphi_X(\theta)^2$ is larger or smaller than 1, which differs according to distribution. This question is a weakened version of the functional equation $g(x+y)=g(x)g(y)$ which has well-known exponential solution. Therefore characteristic functions which decay slower than exponentially are a source of counterexamples, e.g. taking $X$ and $Y=-X$ having Laplace distribution $\mathcal Lap(0,c)$ yields positive acdf for $\theta > \sqrt{2}/c$. For this reason $c_+^\theta$ is better suited to quantify positive dependence and $c_-^\theta$ negative one. Symmetric codifference $s^\theta$ is more robust in this regard.

At last, property i) provides another link between the scdf and covariance, guaranteeing that the former has the same sign as covariance of variables $\sin(\theta X)$ and $\sin(\theta Y)$. These are highly non-linear transformations of $X,Y$, but behave linearly for the part of the probability mass concentrated around 0 as $\sin(x)\approx x$. See also Fig. \ref{fig:sinsin}.

\begin{figure}[H]\centering
\includegraphics[width=8cm]{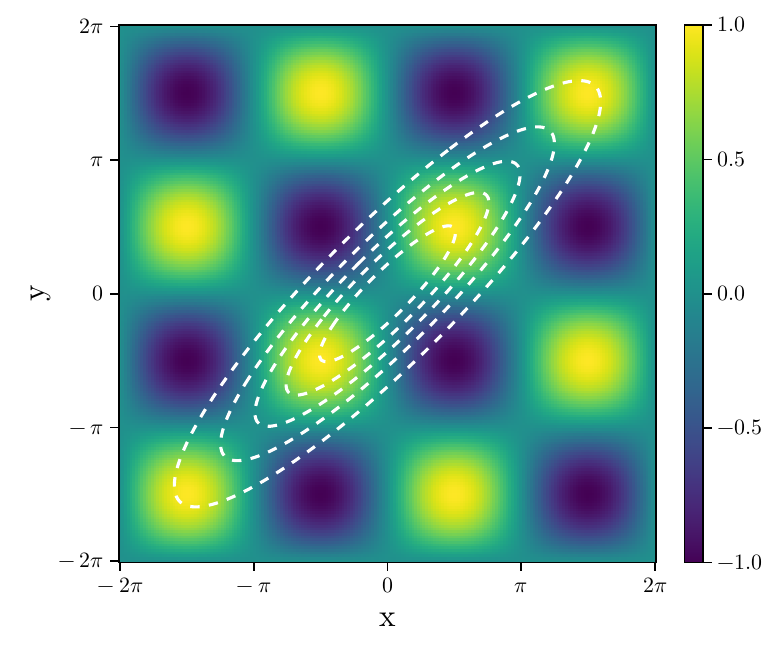}
\caption{The background is the function $\sin(x)\sin(y)$; the white curves are isolines of an exemplary probability density with positive dependence. The plot illustrates how in such a case the majority of the probability mass is located over areas with positive $\sin(x)\sin(y)$.}\label{fig:sinsin}
\end{figure}

\begin{rem}[Usefulness of the dynamical functional]
Looking at Prop. \ref{prp:basicProp} one might notice that the dynamical functional does not have most of the desirable properties listed there. This quantity was introduced in the literature concerning infinitely divisible processes \cite{podgorski91,magdziarz11, loch16,loch18,loch19}. For this class $d^\theta(X_0,X_t)\to 0$ for a one given $\theta\neq 0$ is a sufficient condition for mixing and ergodicity. It is a bounded function (in general it is a complex number $|d^\theta(X,Y)|< 2$, under our assumptions $-1<d^\theta(X,Y)<1$) so it may be viewed as a non-linear analogue of the correlation. Because it only consists of adding and subtracting bounded expected values its estimators behave regularly and its statistical uncertainty is relatively small.

Alas, these practical merits do not change the fact that dynamical functional is hard to interpret outside its original purpose of detecting ergodicity breaking. Like acdf it may be positive for $Y=-X$, it changes in a complex manner when the data is noisy (it gets rescaled, but only if the noise is white) and for Gaussian variables it is a non-linear and non-symmetric $\propto \e^r-1$ transformation of the covariance. Taking all of this into account we do not consider it to be suitable general tool, but it still has worth in its more specific domain.
\end{rem}

\begin{prp}[Reference levels for the codifference] If $[X,Y]\in\mathcal D$ than the following inequalities hold
\begin{itemize}
\item[a)]
\begin{align}
c_+^\theta(X,Y) &\le \frac{c_+^\theta(X,X)+c_+^\theta(Y,Y)}{2}\le \max\{c_+^\theta(X,X),c_+^\theta(Y,Y)\},\\
c_-^\theta(X,Y) &\ge \frac{c_-^\theta(X,-X)+c_-^\theta(Y,-Y)}{2}\ge \min\{c_-^\theta(X,-X),c_-^\theta(Y,-Y)\},
\end{align}
\item[b)]
\begin{equation}
d^\theta(X,Y) \le 1-\sqrt{1-d^\theta(X,X)}\sqrt{1-d^\theta(Y,Y)}\le \max\{d^\theta(X,X),d^\theta(Y,Y)\}.
\end{equation}

\item[c)]
\begin{equation}
-\max\{s^\theta(X,X),s^\theta(Y,Y)\}\le s^\theta(X,Y) \le \max\{s^\theta(X,X),s^\theta(Y,Y)\},
\end{equation}
\end{itemize}
The equalities are reached only for $Y = X$ or $Y = -X$ for $c^\theta_-$.
\end{prp}

\begin{proof}
Points a) and b) are a simple consequence of the inequality $l^\theta(X\pm Y) \ge 0$ and the equality is reached only when $l^\theta(X\pm Y) = 0$ which for $[X,Y]\in\mathcal D$ is possible only if $X = \mp Y$.

Point c) is more interesting. As usual, proving for $\theta = 1$ is sufficient. Due to the symmetry with respect to $X\mapsto -X$ we only need to check the right inequality. Let us assume that $s(X,Y) > \max\{s(X,X),s(Y,Y)\}$. Using the non-negativity of all the expected values under consideration and the fact that the natural logarithm is an increasing function we obtain the inequality
\begin{equation}
\frac{\EE{\cos(X-Y)}}{\EE{\cos(X+Y)}} > \max\left\{\frac{1}{\EE{\cos(2X)}},\frac{1}{\EE{\cos(2Y)}}\right\};
\end{equation}
as a consequence
\begin{equation}
\EE{\cos(2X)} > \frac{\EE{\cos(X+Y)}}{\EE{\cos(X-Y)}} \ge\EE{\cos(X+Y)}.
\end{equation}
and similarly
\begin{equation}
\EE{\cos(2Y)} > \frac{\EE{\cos(X+Y)}}{\EE{\cos(X-Y)}} \ge\EE{\cos(X+Y)}.
\end{equation}
After adding the above two we obtain 
\begin{equation}
\EE{\cos(2X)+ \cos(2Y)} > \EE{2\cos(X+Y)}
\end{equation}
or
\begin{equation}
\EE{\cos(X+Y)\cos(X-Y)} > \EE{\cos(X+Y)}
\end{equation}
Denoting $A\coloneqq X+Y,B\coloneqq X-Y$ this implies that
\begin{equation}\label{eq:condCos}
0 > \EE{\cos(A)(1-\cos(B))}= \EE{\EE{\cos(A)|B}(1-\cos(B))}.
\end{equation}
But, as $[X,Y]\in\mathcal D$, also $[A,B]\in \mathcal D$ and the conditioned variable $A|B\in\mathcal D$; to see this note that $\varphi_{A,B}(\theta_1,\theta_2) = \varphi_{X,Y}(\theta_1+\theta_2,\theta_1-\theta_2)>0$ and $\varphi_{A|B}(\theta) = \varphi_{A,B}(\theta,0)/p_B(B)>0$. Therefore $\EE{\cos(A)|B}>0$ and the whole left hand side of \eqref{eq:condCos} is non-negative, a contradiction. It also shows that the equality is reached only for $X=Y$ for which $\cos(B) = 1$ everywhere.
\end{proof}
\begin{rem}[Strength of the inequalities]
These inequalities are much weaker than Cauchy-Schwartz inequality typically used for covariance. For example, $s^\theta(X,0)=0$, but the bounds are $\pm s^\theta(X)$ whereas Cauchy-Schwartz bound is 0.
\end{rem}

Knowing that the lcf underweights the tails the same should be true for the codifference. An even stronger result is available.
\begin{prp}
    Using decomposition of $X$ into bulk and tails as in Prop. \ref{prp:bulk} only the bulk matters in the far-tail limit, that is
    \begin{equation}
        s^\theta(X,Y)\xrightarrow{w\to\infty} s^\theta(X_\mathrm{bulk},Y), \quad c_\pm^\theta(X,Y)\xrightarrow{w\to\infty} c_\pm^\theta(X_\mathrm{bulk},Y)
    \end{equation}
    and you can also pass to this limit with both $X$ and $Y$ decomposed.
\end{prp}
\begin{proof}
    It follows the same argument as Prop. 1. Because both codifferences depend on each argument once with $+$ and once with $-$, the additional term which appears for the lcf cancels out.
\end{proof}

To end this section we show two elementary examples of variables for which codifference yields interesting information.

\begin{ex} Strong symmetries of scdf make it often agreeing in sign with covariance, but there are cases where it detects non-linear dependence which covariance cannot. Let us take some $A\in\mathcal D$ and $X\coloneqq A, Y \coloneqq  BA$ for $B\indep A$ and $\E[B]=0$. For such pair $\cov(X,Y) =\E[A^2]\E[B]=0$, however scdf measures the difference in scale between $(1-B)A$ and $(1+B)A$. Unless $B\deq - B$ these two are different and $s^\theta(X,Y)\neq 0$. The result depends on details of $A$ and $B$, but in the cumulant expansion \eqref{eq:cumulants} the first non-zero term is $-4\theta^2\E[B^3]\E[A^4]$. The minus sign might be surprising, but becomes more intuitive when one notes that the third moment is dominated by the mass in the tails. $\E[B]=0$ and $\E[B^3]>0$ forces the mass to be in the right tail and left bulk, the latter is what dominates codifference.
\end{ex}

\begin{ex}[Random sum] Here we consider a sum with random number of elements; numerous examples of such quantities appear in financial mathematics and the theory of random walks. Given random summands $X_i$ and $Y_i$ $\indep N$ we define
\begin{equation}
    \mathcal X\coloneqq \sum_{i=1}^N X_i,\quad \mathcal Y\coloneqq \sum_{i=1}^N Y_i
\end{equation}
where the simplest case corresponds to $N\sim\mathcal Poiss(\lambda)$, e.g. for the Markovian random walk $N\sim \mathcal Poiss(\lambda t)$; result for more general class of processes is described later in Prop. \ref{prp:levy}. If $(X_i)_i\indep (Y_i)_i$ the sums $\mathcal X$ and $\mathcal Y$ are uncorrelated $\cov(\mathcal X,\mathcal Y)=0$, but they are clearly dependent through shared $N$. For symmetric distributions also $s^\theta(\mathcal X,\mathcal Y)=0$ due to $[\mathcal X, \mathcal Y]\deq [\mathcal X,-\mathcal Y]$, but the asymmetric codifference is
\begin{equation}
    c_\pm^\theta(\mathcal X,\mathcal Y) = \pm\frac{\lambda}{\theta^2}(1-\varphi_X(\theta))(1-\varphi_Y(\theta))\overset{\varphi_X=\varphi_Y}{=} \pm\frac{\lambda}{\theta^2}(1-\varphi_X(\theta))^2
\end{equation}
and can detect the shared state.

For dependent iid summands $X_i$ and $Y_i$ $s^\theta$ is also non-zero and equals
\begin{equation}
    s^\theta(\mathcal X,\mathcal Y) = \frac{\lambda}{2\theta^2}\left(\varphi_{X_1-Y_1}(\theta)-\varphi_{X_1+Y_1}(\theta)\right)
\end{equation}
and behaves quite differently from the covariance, converging to zero for large-scale $X_1$ and $Y_1$; this is caused by the $\mathcal X = 0,\mathcal Y = 0$ component which then dominates the result. Taking for example $X_1,Y_1\sim\mathcal N\left(0,\begin{bmatrix}1&\rho\\\rho &1\end{bmatrix}\right)$ we get
\begin{equation}
    s^\theta(\mathcal X,\mathcal Y) = \frac{\lambda\e^{-\theta^2}}{\theta^2}\sinh(\rho\theta^2)
\end{equation}
and we see in this case strong correlations are amplified due to $\sinh$ being concave and above line $y=x$ for $x\ge 0$.
\end{ex}

\section{Codifference for mixture models}

A mixture model is the simplest to understand as a distribution which density is a weighted sum/integral over component densities. Thus its characteristic function is also a weighted sum/integral of component characteristic functions and consequently if components belong to the codifference domain $\mathcal D$ the mixture also does. The converse is not generally true.

\subsection{General mixtures}\label{s:genMix}
We will start with discussing how the discussed measures behave for a general, but finite mixture. By that we understand a variable $X$ which is equal to one of the variables $X_1,X_2,\ldots, X_n, X_k\in\mathcal D$ with probabilities $p_1,p_2,\ldots,p_n$ respectively. The resulting probability distribution $\rho$ is a finite sum 
$\rho_X(x) = \sum_k p_k\rho_{X_k}(x)$; similarly for the characteristic function $\varphi_X(\theta) = \sum_kp_k\varphi_{X_k}(\theta)$. From the latter stems the form of the lcf

\begin{equation}
    l^\theta(X) = -\frac{2}{\theta^2}\ln\left(\sum_k p_k \e^{-\frac{\theta^2}{2} l^\theta(X_k)}\right),
\end{equation}
which is a quasi-mean of the component lcfs based on the function $x\mapsto\exp(-\theta^2x/2)$. As described in Remark \ref{rem:qm} the lcf is a quasi-mean for any distribution, but the reduced form for the mixtures has as a base a monotonic function, which makes its behaviour much more regular.

The function $\ln\left(\sum_k \exp(x_k)\right)$ is also known as the log-sum-exp function which is used as a smooth approximation of the maximum in optimisation and machine learning problems \cite{mathOpt}. This suggests that the lcf for mixtures is somewhat between the arithmetic mean and minimum (due to the minus) functions. For a comparison, variance behaves arithmetically, $\var(X) = \sum_k p_k\var(X_k)$, and the $L^q$ norm $(\E[|X|^q])^{1/q}$ for $q=1$ is arithmetic, becoming like maximum for large $q$.

\begin{prp}[Mean inequality for the lcf]
    The following inequality holds for any mixture $(p_k,X_k)_{k=1}^n$
    \begin{equation}
        \min_k l^\theta(X_k)\le l^\theta(X) \le \sum_k p_k l^\theta(X_k).
    \end{equation}
\end{prp}
\begin{proof}
It follows from the inequality
   \begin{equation}
    \e^{-\frac{\theta^2}{2}\sum_k p_k l^\theta(X_k)} \le \sum_k p_k \e^{-\frac{\theta^2}{2} l^\theta(X_k)}\le \e^{-\frac{\theta^2}{2} \min_k l^\theta(X_k)}.
\end{equation} 
The left hand side is Jensen inequality. The right hand side is just $\exp(-x)$ being decreasing.
\end{proof}

So, the lcf stresses the components with smaller spread. We have already seen this behaviour in decomposition into bulk and tails (Prop. \ref{prp:bulk}). This behaviour goes further, the smallest component completely dominates the result for spread-out mixtures.

\begin{prp}[The lcf for large-scale mixtures]\label{prp:ls}
Let $X$ be drawn from a mixture $(p_k,X_k)_{k=1}^n$ for which $\varphi_{X_k}= o(\varphi_{X_1})$ at $\infty$, $k>1$. Then for $w X$ at limit $w\to\infty$ only the first component counts,
\begin{equation}
    l^\theta(w X)-l^\theta(wX_1)\xrightarrow{w\to\infty}-\frac{2}{\theta^2}\ln p_1.
\end{equation}
\end{prp}
\begin{proof}
    Case $\theta=1$ is sufficient. By straightforward simplification
    \begin{equation}
        l(w X)-l(wX_1) =-2\ln\left(\sum_k p_k\varphi_{X_k}(w)\right) + 2 \ln \varphi_{X_1}(w) = -2\ln\left(\sum_k p_k\frac{\varphi_{X_k}(w)}{\varphi_{X_1}(w)}\right)\xrightarrow{w\to\infty}-2\ln p_1.
    \end{equation}
\end{proof}
In terms of the lcf the condition $\varphi_X\in o(\varphi_Y)$ at $\infty$ can be rewritten as $l^\theta(w X)-l^\theta(w Y)\to\infty, w\to\infty$. For example, for Gaussian variables the one with the lower variance will dominate. For any $X\indep Y$ variable $X$ will dominate $X+Y$ as $\varphi_{X+Y}/\varphi_X = \varphi_Y\to 0$ at $\infty$.

A natural next question is what happens for small-scale mixtures. Partially we know the answer already: at small scales the lcf becomes variance (Prop. \ref{prp:lcf} e)). So, something more interesting can happen only for variables with infinite second moment.

\begin{prp}[The lcf for small scale mixtures]\label{prp:ss}
Let $X$ be drawn from a mixture $(p_k,X_k)_{k=1}^n$ for which $1-\varphi_{X_k}\in o(1-\varphi_{X_1})$ at $0$, $k>1$. Then for $\epsilon X$ at limit $\epsilon\to0$ only the first component counts,
    \begin{equation}
        l^\theta(\epsilon X)\sim p_1l^\theta(\epsilon X_1) \sim \frac{2p_1}{\theta^2}\big(1-\varphi_{X_1}(\epsilon \theta)\big),\quad \epsilon\to 0.  
    \end{equation}
\end{prp}

\begin{proof}
For $\theta = 1$ the lcf reads $l(\epsilon X) = -2\ln\varphi_X(\epsilon)$. Expanding the logarithm we get
    \begin{equation}
         -2\ln \varphi_X(\epsilon)\sim -2\big(\varphi_X(\epsilon)-1)\sim-2\sum_kp_k\big(\varphi_{X_k}(\epsilon)-1\big) \sim  -2 p_1(\varphi_{X_1}(\epsilon)-1),\quad \epsilon\to 0.
    \end{equation}
\end{proof}
Close-to-zero asymptotic of the characteristic function is closely related to the tails of distribution. For example, symmetric stable variables have $\varphi_X(\theta) = \exp(-|c\theta|^\alpha), 0\le \alpha \le 2$. In a small-scale mixture the component with the smallest $\alpha$ will dominate. 

\begin{rem}[Parameter $\theta$ and scale]
    In Props. \ref{prp:ls} and \ref{prp:ss} we kept $\theta$ fixed and changed scale, but for a fixed sample the scale can be changed through $\theta$ itself as variables appear in the lcf and codifference multiplied by it. For small $\theta$ these quantities become like variance/covariance or diverge, signalising the presence of power law tails. At large $\theta$ it is the opposite: they start ignoring tails more and more, measuring the properties of a given distribution only close to zero. 
\end{rem}

A natural complement to the above results is considering the question what happens if we keep the scale, but a mixture has a dominating component and the rest is perturbation in terms of mass.
\begin{prp}[Mixture with high probability component] Let $X$ be a two component mixture which is $X_1$ with probability $1-\epsilon$ and $X_2$ with probability $\epsilon$. Then we have  asymptotic expansion
\begin{equation}
    l^\theta(X)\sim l^\theta(X_1)+\frac{2}{\theta^2}\sum_{n=1}^\infty \frac{(-\epsilon)^n}{n}\left(\frac{\varphi_{X_2}(\theta)}{\varphi_{X_1}(\theta)}-1\right)^n,\quad \epsilon\to0^+.
\end{equation}
\end{prp}
\begin{proof} It follows from writing $l^\theta(X) = -2\theta^{-2}\ln\big(\varphi_{X_1}(\theta) +\epsilon(\varphi_{X_2}(\theta)-\varphi_{X_1}(\theta))\big)$, taking out $\varphi_{X_1}(\theta)$ and expanding $\ln(1+\epsilon(\ldots))$.
\end{proof}
For $X_2$ with more spread than $X_1$, $\varphi_{X_1}(\theta)>\varphi_{X_2}(\theta)$ and the first term correction is positive. In this limit adding more spread prevails over removing mass from the bulk.

To finish this section let us explain how the above facts translate to the codifference. It is a difference of two or three lcf (for $s^\theta$ or $c_\pm^\theta$). The scdf again is closer to covariance because the terms $\ln p_1$ cancel out, so it has form covariance + decreasing corrections.

An added nuance is that in some situations different components can win for different lcf terms $l^\theta(X), l^\theta(Y), l^\theta(X+Y), l^\theta(X-Y)$. To illustrate this phenomenon, one example is $Y_1 = -X_1 + Z$, $Y_2 = +X_2+Z$ $X_{1,2}$ being Cauchy variables and $Z$ being Gaussian. In a small-scale mixture for $l^\theta(X-Y)$ the first pair will win, for $l^\theta(X+Y)$ the second pair will win. So in a sense, codifference measures maximal spread of concentration along $X = \pm Y$ among all components.

\subsection{Gaussian mixtures}\label{s:gaussMix}
Gaussian mixtures are arguably the most common type of mixture models, so it is worth discussing them in more detail. These models assume that the mean and covariance of Gaussian distribution can be random. Making mean random is equivalent to adding random variable, so in our case it is not particularly interesting. In the case of random covariance we start with the simplest case of only variance random, which can be represented as $\sqrt{S}X, S\indep X$. Results from Sec. \ref{s:genMix} apply, so for small scale mixtures the lcf and codifference become like variance as long as $\E[S]<\infty$. Corrections can be traced using cumulant expansion \eqref{eq:cumulants}.

For large scale mixtures Prop. \ref{prp:ls} applies, but knowing the exact decay of Gaussian characteristic function, the full asymptotic expansion of lcf is also achievable. The convergence is very fast.
\begin{prp}[The lcf for large-scale discrete Gaussian mixtures]\label{prp:lsdG}
Let $\pr(S=\sigma_k^2) = p_k, \sigma_1<\sigma_2<\ldots, \E[S]<\infty, S\indep X, X\sim \mathcal N(0,1)$. Then $l^\theta$ has asymptotic expansion
\begin{equation}
    l^\theta(w \sqrt{S}X)\sim w^2\sigma_1^2-\frac{2}{\theta^2}\ln p_1 -\frac{2}{p_1\theta^2} \sum_{k=2}^\infty p_k\e^{-w^2(\sigma_k^2-\sigma_1^2)\theta^2/2},\quad w\to\infty.
\end{equation}
\end{prp}
\begin{proof}
For simplicity we denote $\alpha_k\coloneqq (\sigma_k^2-\sigma_1^2)\theta^2/2\ge 0$. Having $\E[S]<\infty$ implies $\sum_k p_k\alpha_k<\infty$. Subtracting first $n$ terms of the expansion form $l^\theta$ we get
\begin{equation}
    l^\theta(w \sqrt{S}X)- w^2\sigma_1^2+\frac{2}{\theta^2}\ln p_1 +\frac{2}{p_1\theta^2} \sum_{k=2}^n p_k\e^{-w^2\alpha_k} = -\frac{2}{\theta^2}\left(\ln\left(1+\frac{1}{p_1}\sum_{k=2}^\infty p_k \e^{-w^2\alpha_k}\right)-\frac{1}{p_1}\sum_{k=2}^n p_k\e^{-w^2\alpha_k}\right).
\end{equation}
The right side is $o(\e^{-w^2\alpha_n})$ which is apparent using L'Hôpital's rule; differentiating infinite series is allowed as $\sum_k p_k\alpha_k<\infty$.
\end{proof}
What happens when $\E[S]=\infty$ is that we can then treat the tail $\sigma_{n+1},\sigma_{n+2},\ldots$ as a single variable $Y$ which may have a slower decaying characteristic function dominating $\e^{-w^2\alpha_k}, k\le n$.

A natural following question is what happens for a large-scale mixing distribution $S$ which has density $\rho$. Interestingly, its behaviour is quite universal $\sigma_0^2 w^2+ C \ln w$.

\begin{prp}[The lcf for large-scale continuous Gaussian mixtures]\label{prp:lsG} Let $S\indep X,X\sim\mathcal N(0,1)$ and $S$ have a density $\rho$ on $[\sigma_0^2,\infty)$ which has asymptotics $\rho(\sigma_0^2+s)\sim L(s)s^\alpha,s\to 0^+$, $L$ is slowly varying. Then
\begin{equation}
    l^\theta(w\sqrt{S}X)-w^2\sigma_0^2-\frac{2}{\theta^2}\left(2(1+\alpha)\ln w-\ln(\Gamma(1+\alpha) L(w^{-1})\right)+(1+\alpha)\ln(\theta^2/2)\xrightarrow{w\to\infty} 0.
\end{equation}
\end{prp}
\begin{proof} Notice that $\ln \varphi(w) - \ln f(w)\to 0$ iff $\varphi\sim f$. But $\varphi$ is the Laplace transform of $\rho$
\begin{equation}
    \EE{\e^{-(w^2\theta^2/2)S}}= \e^{-\sigma_0^2w^2\theta^2/2}\int_{0}^\infty \rho(\sigma_0^2+s)\e^{-(w^2\theta^2/2)s}\dd s.
\end{equation}
From Tauberian theorem the integral on the right is $\sim \Gamma(1+\alpha)\left(\frac{w^2\theta^2}{2}\right)^{-1-\alpha} L(w^{-1})$. Calculating the logarithm of it yields the result.
\end{proof}
This shows that in agreement with our previous remarks, the lcf ignores the tails and measures the small component shape of the mixture.

\begin{rem}[Codifference for large-scale continuous Gaussian mixtures]
    From Prop. \ref{prp:lsG} limits of the codifference follow. For Gaussian $X,Y$ with large covariance matrix $\Sigma$ the leading term of $s^\theta(\sqrt{S}X,\sqrt{S}Y)$ is the covariance $\Sigma_{XY}\sigma_0^2$  and the next term is logarithmic
    \begin{equation}
        \propto (1+\alpha)\ln\frac{\Sigma_{XX}+2\Sigma_{XY}+\Sigma_{YY}}{\Sigma_{XX}-2\Sigma_{XY}+\Sigma_{YY}},
    \end{equation}
    similarly for $c_\pm^\theta$. It does not depend on any other details of the mixing distribution $S$, reflecting the fact that in this regime the observed mixture is asymptotically self-similar.
\end{rem}

\begin{ex}[Variance gamma mixture] A common model of fat tail data is a variance gamma distribution \cite{varGamma} which is (in the symmetric case) a Gaussian mixture with variance $S\sim\Gamma(\alpha,1/2)$. It is a nice illustration of the results above. For $X,Y \sim\mathcal N\left(0,\begin{bmatrix}1& \rho\\\rho &1\end{bmatrix}\right)$ we have
\begin{equation}
    s^\theta(X,Y) = \frac{\alpha}{2\theta^2}\ln\frac{1+2\theta^2(1+\rho)}{1+2\theta^2(1-\rho)},\quad c^\theta_\pm(X,Y) = \frac{\alpha}{\theta^2}\ln\frac{(1+\theta^2)^2}{1+2\theta^2(1\mp\rho)}
\end{equation}
and we see $\alpha$, while changing the shape of the distribution considerably, only rescales the codiffference.
\end{ex}

It may also be of use to express all the considered functions explicitly using cumulant-generating function.

\begin{prp}[Codifference and the lcf for the classical mixture model]\label{prp:GaussianHet} Let $X,Y$ have joint Gaussian distribution with covariance matrix $\Sigma$ and $S$ be a non-negatve random variable, $S\indep [X,Y]$, with Laplace transform
\begin{equation}
\EE{\e^{-\theta S}} = \e^{-\psi(\theta)},
\end{equation}
that is, $\psi$ is minus cumulant-generating function of $-S$. Then the discussed measures of $X'=\sqrt{S}X, Y'=\sqrt{S}Y$ are:
\begin{itemize}
    \item[a)] lcf
    \begin{equation}
    l^\theta(X') = \frac{2\psi(\theta^2\Sigma_{XX}/2)}{\theta^2}
    \end{equation}
    \item[b)] scdf
    \begin{equation}
    s^\theta(X',Y')= \frac{1}{2\theta^2}\left(\psi(\theta^2(\Sigma_{XX}/2+\Sigma_{YY}/2+\Sigma_{XY})) - \psi(\theta^2(\Sigma_{XX}/2+\Sigma_{YY}/2-\Sigma_{XY}))\right).
    \end{equation}
    \item[c)] acdf
    \begin{equation}
    c_\pm^\theta(X',Y') = \pm\frac{1}{\theta^2}\left(\psi(\theta^2\Sigma_{XX}/2)+\psi(\theta^2\Sigma_{YY}/2) -\psi(\theta^2(\Sigma_{XX}/2+\Sigma_{YY}/2\mp\Sigma_{XY}))\right),
    \end{equation}
    \item[d)] dynamical functional
    \begin{equation}
    d^\theta(X',Y') = \e^{-\psi(\theta^2(\Sigma_{XX}/2+\Sigma_{YY}/2-\Sigma_{XY}))}-\e^{-\psi(\theta^2\Sigma_{XX}/2)-\psi(\theta^2\Sigma_{YY}/2)},
    \end{equation}

\end{itemize}
\end{prp}

\begin{proof} 
The result follows directly from the conditional expectancy and the characteristic function of Gaussian variables
\begin{equation}
\EE{\e^{\I \sqrt{S}G}\big|S} = \e^{-S\EE{G^2}/2}
\end{equation}
where next we expand $\E[G^2]$ for each term using the covariance matrix.
\end{proof}

Predictably, codifference has a simple behaviour for decompositions of mixture distributions.

\begin{prp}
    Let $S = S_1+S_2 ,S_1\indep S_2, [S_1,S_2]\indep [X,Y], [X,Y]\sim\mathcal N(0,\Sigma)$. Then for a Gaussian mixture
    \begin{align}
    l^\theta(\sqrt{S}X) &= l^\theta(\sqrt{S_1}X)+l^\theta(\sqrt{S_2}X),\nonumber\\
        s^\theta(\sqrt{S}X,\sqrt{S}Y) &= s^\theta(\sqrt{S_1}X,\sqrt{S_1}Y) +s^\theta(\sqrt{S_2}X,\sqrt{S_2}Y) \nonumber\\
        c_\pm^\theta(\sqrt{S}X,\sqrt{S}Y) &= c_\pm^\theta(\sqrt{S_1}X,\sqrt{S_1}Y) +c_\pm^\theta(\sqrt{S_2}X,\sqrt{S_2}Y)
    \end{align}
\end{prp}
\begin{proof}
    This stems from the fact that $l^\theta$ depends linearly on $\psi$ for which $\psi_S = \psi_{S_1}+\psi_{S_2}$; codifferences $c^\theta_\pm,s^\theta$ depend linearly on $l^\theta$.
\end{proof}

A further consequence is that it is reasonably simple to use codifference for  mixture-like perturbations of Gaussianity. For $S = 1 +\epsilon S'$ we have $l^\theta(\sqrt{S}X)= 1 + l^\theta(\epsilon\sqrt{S'}X)$, similarly for $s^\theta$ and $c^\theta_\pm$. These are small scale mixtures with behaviour dominated by first cumulants.

For more general mixtures with random covariance matrix, models need to be studied on a more case-to-case basis. Still, some general properties can be established in the limits of strong and weak memory. We show them for symmetric codifference and $\theta=1$, $\theta\neq 1$ can be easily recovered rescaling the covariance matrix. Formulas for asymmetric codifference can be obtained in analogical way, but the formulas are noticeably uglier.
\begin{prp}
Asymptotic codifference for conditionally Gaussian vectors.
\begin{itemize}
    \item[a)] (Weak dependence.) Let the random covariance matrix be parametrized as
    \begin{equation}
    \Sigma = \bgmx
      S_1 & \varepsilon R\\
      \epsilon R & S_2
      \emx,
      \end{equation}
      then the codifference has the small $\varepsilon$ asymptotic
      \begin{equation}
      s(X,Y) =2\varepsilon\frac{\EE{R \e^{-(S_1+S_2)/2}}}{\EE{\e^{-(S_1+S_2)/2}}}+\varepsilon^2\frac{\EE{R\e^{-(S_1+S_2)/2}}^2}{\EE{\e^{-(S_1+S_2)/2}}^2}+\mathcal{O}(\varepsilon^3).
      \end{equation}
    \item[b)] (Strong dependence.) For random covariance matrix close to the complete correlation
    \begin{equation}
    \Sigma = \bgmx
      S &  \alpha S-\varepsilon W\\
      \alpha S-\varepsilon W  & \alpha^2 S
      \emx,
      \end{equation}
      where $\alpha$ is deterministic number and $\EE{|W|}<\infty$,	we have the small $\varepsilon$ asymptotic
      \begin{align}
      &s(X,Y)= s(X,\alpha X) - \varepsilon \frac{\EE{W\e^{-(1-\alpha)^2 S/2}}}{\EE{\e^{-(1-\alpha)^2S/2}}} \nonumber\\
      &+\varepsilon^2\left(\frac{\EE{W^2\e^{-(1-\alpha)^2S/2}}}{\EE{\e^{-(1-\alpha)^2S/2}}}+\frac{\EE{W\e^{-(1-\alpha)^2S/2}}}{\EE{\e^{-(1-\alpha)^2S/2}}}\left(2\frac{\EE{W\e^{-(1+\alpha)^2S/2}}}{\EE{\e^{-(1+\alpha)^2S/2}}} - \frac{\EE{W\e^{-(1-\alpha)^2S/2}}}{\EE{\e^{-(1-\alpha)^2S/2}}}\right) \right)\nonumber\\
      &+\mathcal{O}(\varepsilon^3).
      \end{align}
      Additionally, the above formula simplifies to
      \begin{equation}
   	      s(X,Y)= s(X,\alpha X) - \varepsilon\EE{W}\nonumber\\
      +\varepsilon^2\left(\EE{W^2}+\EE{W}\left(2\frac{\EE{W\e^{-2S}}}{\EE{\e^{-2S}}} -\EE{W}\right) \right)\nonumber\\
      +\mathcal{O}(\varepsilon^3)
      \end{equation}
      for $\alpha=\pm1$, that is $X$ and $Y$ close to being equal/opposite, with the additional assumption $\EE{W^2}<\infty$.
\end{itemize}
\end{prp}

\begin{proof}
Calculation very similar to Prop. \ref{prp:GaussianHet} which shows that for Gaussian mixture with random covariance matrix the symmetric codifference can be expressed as
\begin{equation}
s(X,Y) = \ln\left(1+2\frac{\EE{\sinh(\Sigma_{XY})\e^{-(\Sigma_{XX}+\Sigma_{YY})/2}}}{\EE{\e^{-\Sigma_{XY}}\e^{-(\Sigma_{XX}+\Sigma_{YY})/2}}}\right).
\end{equation}
For a) it has a form $\ln(1+f(\varepsilon)/g(\varepsilon))$ with
\begin{equation}
f(\varepsilon) = 2\EE{\sinh(\varepsilon R)\e^{-(S_1+S_2)/2}},\quad g(\varepsilon)  = \EE{\e^{-\varepsilon R}\e^{-(S_1+S_2)/2}}.
\end{equation}
We can switch the order of differentiation and expectation because variable $R$ is bounded by $-\sqrt{S_1S_2}<R<\sqrt{S_1S_2}$, and function $x^\gamma\e^{-x},\gamma>0$ is bounded and integrable. Using $f(0)=0$ the Taylor expansion of log can be expressed in a simplified form
\begin{equation}
\ln\left(1+\frac{f(\varepsilon)}{g(\varepsilon)}\right)=\varepsilon \frac{f'(0)}{g(0)}+\varepsilon^2\frac{g(0)f''(0)-f'(0)(f'(0)+2g'(0))}{2g(0)^2}+\mathcal{O}(\varepsilon^3)
\end{equation}
The last term further simplifies  because $f''(0)=0$.

For b) we can express $s(X,Y)-s(X,\alpha X)$ in the same manner with
\begin{equation}
f(\varepsilon) = \frac{\EE{\left(\e^{-\alpha\varepsilon W}-1\right)\e^{-(1-\alpha)^2S/2}}}{\EE{\e^{-(1-\alpha)^2 S/2}}}, \quad g(\varepsilon) = \frac{\EE{\e^{\alpha\varepsilon W}\e^{-(1+\alpha)^2 S/2}}}{\EE{\e^{-(1+\alpha)^2 S/2}}}.
\end{equation}
Commuting differentiation and integration is again allowed because of the boundedness in the case $\alpha\neq \pm 1$, for the special case $\alpha = \pm1$ we need to additionally require $\E[W^2]<\infty$.

\end{proof}
Additionally, as we see from the first expansion terms depend on products of $R$ or $W$ and exponentials of $S_1,S_2$, so any independent factors can be pushed outside the expectancy.

One can also look at the case of strong dependence as $X-Y$ or $X+Y$ being small compared to the other. In this setting the analysis is simple as one needs to consider the lcf function for a small variable, a procedure which we discussed before.

\begin{ex}[Cross-shaped distribution] A good example of distribution which would break classical measures of dependence is a mixture of two Gaussians which looks like a cross, see Fig. \ref{fig:cross}. Let $X,Y$ with probability 1/2 have distribution $\mathcal N\left(0,\begin{bmatrix}1& 0\\0 &\sigma^2\end{bmatrix}\right)$ and with 1/2 distribution $\mathcal N\left(0,\begin{bmatrix}\sigma^2& 0\\0 & 1\end{bmatrix}\right)$, $\sigma^ 2 \ll 1$. The variance of $X$ is $(1+\sigma^2)/2$. We expect the lcf to be smaller as it puts more emphasis into the bulk which here is $\sigma^2$ component. Indeed, it can be expressed as

\begin{equation}
    l^\theta(X)= \frac{1+\sigma^2}{2}-\frac{2}{\theta^2}\ln\cosh\left(\frac{\theta^2}{2}\frac{1+\sigma^2}{2}\right).
\end{equation}
The covariance $\cov(X,Y)$ is just 0, the same is true for $s^\theta$ as $[X,Y]\deq [X,-Y]$. But for asymmetric codifference
\begin{equation}
    c^\theta_\pm(X,Y) = \mp\frac{2}{\theta^2}\ln\cosh\left(\frac{\theta^2}{2}\frac{1+\sigma^2}{2}\right).
\end{equation}
It is negative for $c_+$ and positive for $c_-$. Quantity $c_+$ is closer to the intuitive understanding of dependence as in this model observing one component being large correlates with the second one being smaller; all of this is in respect to the amplitude, so it is strictly non-linear.

The same model rotated by $45^\circ$ results in (up to a scale) mixture 1/2 of $\mathcal N\left(0,\begin{bmatrix}1& \rho\\\rho &1\end{bmatrix}\right)$ and 1/2 of $\mathcal N\left(0,\begin{bmatrix} 1& -\rho\\-\rho & 1\end{bmatrix}\right)$, correlation $\rho$ is close to 1. It is a non-Gaussian distribution with Gaussian marginals. Due to the reflection symmetry again $\cov(X,Y) = s^\theta(X,Y) = 0$, but

\begin{equation}
    c_\pm^\theta(X,Y) = \mp\frac{1}{\theta^2}\ln\cosh\left(\rho\theta^2\right)
\end{equation}
the limit $\rho\to 1$ corresponds to the case $Y = +X$ or $Y=-X$ with equal probabilities.
\end{ex}

\begin{figure}[H]\centering
\includegraphics[width=12cm]{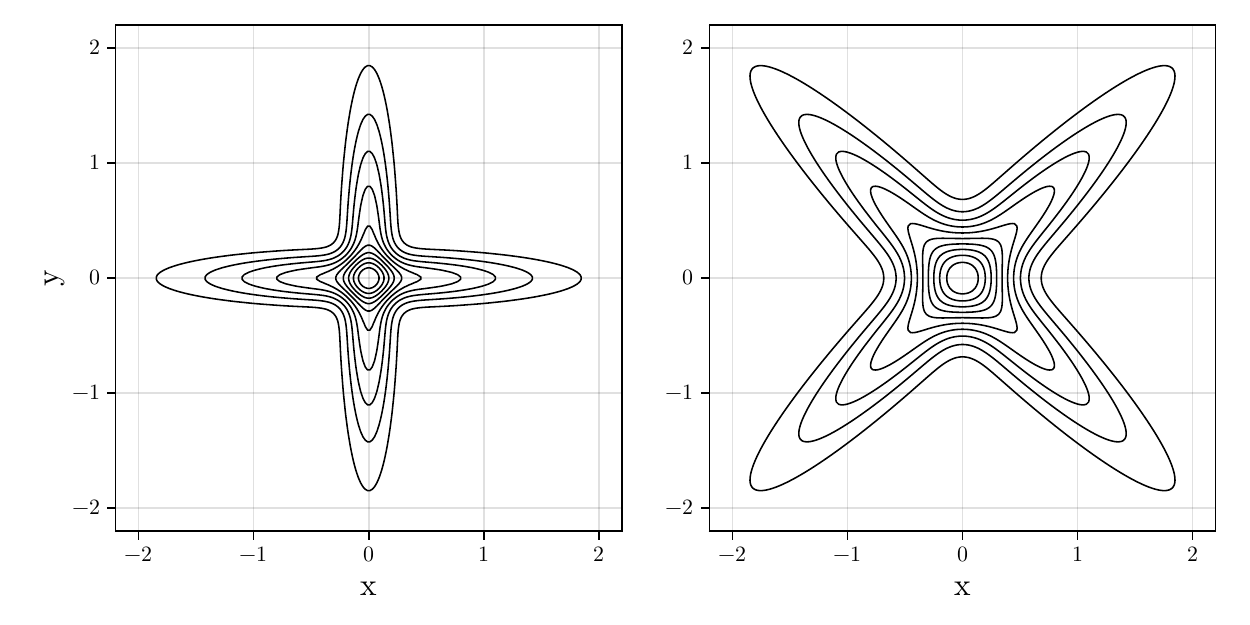}
\caption{Shapes of the exemplary cross-shaped Gaussian mixtures.}\label{fig:cross}
\end{figure}

\begin{ex} For a mixture of two Gaussians, one small and the other large scale, we have already established that the covariance will be dominated by the larger one (unless it is an exactly uncorrelated component) and the codifference on the contrary, by the smaller one. By balancing mixing probabilities and covariances it is easy to produce examples for which covariance and scdf or acdf disagree to an arbitrarily high degree.

Due to its low influence we can also treat the second component as a perturbation. This is not exactly Prop. \ref{prp:lsdG}, but the calculation is nearly identical. Considering $X,Y\sim\mathcal N\left(0,\begin{bmatrix}\sigma_1^2 & c_1\\c_1& \sigma_1^2\end{bmatrix}\right)$ with probability $p$ and $X,Y\sim\mathcal N\left(0,\begin{bmatrix}\sigma_2^2 & c_2\\c_2& \sigma_2^2\end{bmatrix}\right)$ with probability $1-p$ and scaling the second covariance matrix to infinity, we get
\begin{equation}
    s^\theta(X,Y)\sim c_1+\frac{1-p}{p}\e^{-\theta^2(\sigma_2^2-\sigma_1^2)}\frac{1}{\theta^2}\sinh\big(\theta^2(c_2-c_1)\big).
\end{equation}
The correction decreases exponentially with the scale of the second component which again shows that the codifference is a good tool for quantifying small-scale dependence.
\end{ex}

\section{Codifference for stochastic processes}

Here we analyse the usage of codifference few few important types of stochastic processes and we highlight its relation to covariance.

The usual way to measure the temporal dependence of the zero-mean stochastic processes is the covariance function, that is
\begin{equation}
\cov(X_t,X_s)\coloneqq \EE{X_{t}X_{s}}
\end{equation}
which is just the covariance of the vector $[X_{t},X_{s}]$.  This is also a natural definition for the codifference.

\begin{dfn}[Codifference of a stochastic process.] Given a process $(X_t)_t$ the (auto-) codifference function $s_X(s,t)$ is a codifference calculated at vector $[X_{t_2},X_{t_1}]$, that is, explicitly, 
\begin{equation}
s^\theta_X(t_1,t_2)\coloneqq s^\theta(X_{t_2},X_{t_1}),\quad c_{\pm,X}^\theta(t_1,t_2)\coloneqq c_\pm^\theta(X_{t_2},X_{t_1}), \quad d^\theta_X(t_1,t_2)\coloneqq d^\theta(X_{t_2},X_{t_1}) .
\end{equation}
This definition generalises to random fields (for $t\in\mathbb R^n$) and vector-valued processes (for both auto- and cross- codifferences) in a natural way.
\end{dfn}

\begin{prp}[Codifference of a stationary process]\label{prp:stat}
For a statinary stochastic process the lcf is constant and all the codifference variants depend only on the difference of arguments
\end{prp}

\begin{proof}
    This is straightforward. The lcf $l_X^\theta(t)$ is a functional of the distribution of $X_t$, so it must be constant. Vector process $\tau\to [X_{t_1+\tau},X_{t_2+\tau}]$ is also stationary and $s^\theta, c_\pm^\theta, d^\theta$ are functionals of its distribution, so they are also constant, thus $s^\theta(t_1+\tau,t_2+\tau) = \const$, $c^\theta_\pm(t_1+\tau,t_2+\tau)=\const, d^\theta(t_1+\tau,t_2+\tau)=\const$
\end{proof}

\begin{prp}[Codifference of a L\'evy process]\label{prp:levy}
  Let a L\'evy process $(X_t)_{t\ge 0}$ have L\'evy index $\Psi$, that is $\EE{\e^{\I\theta X_t}} = \e^{t\Psi(\theta)}$. Then the codifference of such process is
  \begin{itemize}
      \item[a)] $s_X^\theta(t_1,t_2) = \min\{t_1,t_2\}\frac{\Psi(2\theta)}{2\theta^2}$
      \item[b)] $c_{\pm,X}^\theta(t_1,t_2) =  \min\{t_1,t_2\}\frac{\Psi(\theta)}{\theta^2}$
      \item[c)] $d_X^\theta(t_1,t_2) = \e^{|t_2-t_1|\Psi(\theta)} - \e^{(t_1+t_2)\Psi(\theta)} $
  \end{itemize}
\end{prp}
Using the independent increments property the proof is just a simple calculation, so we skip it. One can compare it to the covariance function which, whenever it exists for a symmetric L\'evy process, equals $-\min\{t_1,t_2\}\Psi''(0)$; codifference here seems to be a natural extension.

\begin{ex}[Stationary process with bursts] Let us consider a two state stationary process $X$ for which the first state is a stationary process $X_1$ observed for time $\mathcal Exp(\lambda_1)$ after which for time $\mathcal Exp(\lambda_2)$ stationary process $X_2$, $X_2\indep X_1$, is observed, then we observe $X_1$ again, and so on. The renewal events also reset the memory. If $X_2$ has a much larger variance than $X_1$, or even infinite variance, we can interpret it as 'bursts'. Such periods of activity are observed for many types of data, a classical example are solar flares, see, e.g. \cite{solarFlare, solarFlare2}.

The covariance of such time series would be dominated by the bursts and is not of much use unless the two states are identified before. But the codifference is finite and can account for both states

\begin{equation}
    s^\theta_X(\tau) = \frac{1}{4\theta^2}\ln \frac{\frac{\lambda_1^{-1}}{\lambda_1^{-1}+\lambda_2^{-1}}\e^{-\lambda_{1}\tau}\varphi_{X_1(0)-X_1(\tau)}(\theta) + \frac{\lambda_2^{-1}}{\lambda_2^{-1}+\lambda_2^{-1}}\e^{-\lambda_{2}\tau}\varphi_{X_2(0)-X_2(\tau)}(\theta)}{\frac{\lambda_1^{-1}}{\lambda_1^{-1}+\lambda_2^{-1}}\e^{-\lambda_{1}\tau}\varphi_{X_1(0)+X_1(\tau)}(\theta) + \frac{\lambda_2^{-1}}{\lambda_2^{-1}+\lambda_2^{-1}}\e^{-\lambda_{2}\tau}\varphi_{X_2(0)+X_2(\tau)}(\theta)}
\end{equation}
in particular the short bursts can be treated as a perturbation
\begin{equation}
  s^\theta_X(\tau) \sim s^\theta_{X_1}(\tau) + \frac{\lambda_1}{\lambda_2}\e^{(\lambda_1-\lambda_2)\tau}\left(\frac{\varphi_{X_2(0)-X_2(\tau)}(\theta)}{\varphi_{X_1(0)-X_1(\tau)}(\theta)}-\frac{\varphi_{X_2(0)+X_2(\tau)}(\theta)}{\varphi_{X_1(0)+X_1(\tau)}(\theta)}\right),\quad \lambda_2\to\infty
\end{equation}
therefore in contrast to the covariance the codifference can be used directly for such data.
\end{ex}

\section{Estimation}
The most common estimator of the characteristic function is the empirical one (ecf), which for a sample $(X_1,X_2,\ldots,X_n)$ reads \cite{feuerverger}
\begin{equation}
    \widehat\varphi_X(\theta)\coloneqq\frac{1}{n}\sum_{k=1}^n \e^{\I\theta X_k}.
\end{equation}
This is just the arithmetic mean of the sample of $\e^{\I\theta X}$, which is bounded by 1. So, the law of large numbers applies and $\widehat \varphi_X\to \varphi_X$ as $n\to\infty$, moreover by the central limit theorem asymptotic deviations from the mean are Gaussian. We will calculate parameters of this distribution in a slightly more general setting.

Let  the set of pairs $(A_k,B_k)$ be an iid sample. We directly calculate the covariance of $\widehat \varphi_A$ and $\varphi_B$. For every $\theta$,
\begin{align}
    \EE{\big(\widehat\varphi_A-\varphi_A\big)\big(\widehat\varphi_B-\varphi_B\big)} &=\frac{1}{n}(\varphi_{A+B}-\varphi_A\varphi_B), \nonumber\\
    \EE{\big(\widehat\varphi_A-\varphi_A\big)\big(\widehat\varphi_B-\varphi_B\big)^*} &=\frac{1}{n}(\varphi_{A-B}-\varphi_A\varphi_B), 
\end{align}
where by $z^*$ we denote complex conjugation. Consequently, for the real part
\begin{equation}\label{eq:realCov}
    \EE{\big(\mathrm{re}\,\widehat\varphi_A-\varphi_A\big)\big(\mathrm{re}\,\widehat\varphi_B-\varphi_B\big)} =\frac{1}{n}\left(\frac{\varphi_{A+B}+\varphi_{A-B}}{2}-\varphi_A\varphi_B\right) \eqqcolon\frac{1}{n}C_{A,B}
\end{equation}
 Substituting $A=B=X$ we obtain the variance of $\widehat \varphi_X$ and the limit $\sqrt{n}\big(\mathrm{re}\,\widehat\varphi_X-\varphi_X\big)\xrightarrow{d}\mathcal N(0,C_{X,X})=\mathcal N\big(0,(\varphi_{2X}+1)/2-\varphi_X^2\big)$ as $n\to\infty$. Asymptotic normality will transfer to the estimators of the discussed measures, which are defined in a natural way.
\begin{dfn}[Empirical estimates of the lcf and codifferences]  
Given estimator ecf, the estimators for lcf and codifferences can be constructed in a natural manner,    

\begin{align}
    \widehat l^{\,\theta}(X)&\coloneqq   -\frac{2}{\theta^2}\ln\big(\mathrm{re}\,\widehat\varphi_X\big),\\
    \widehat s^{\,\theta}(X,Y)&\coloneqq \frac{1}{4}\left(\widehat l^{\,\theta}(X+Y)-\widehat l^{\,\theta}(X-Y)\right),\\
    \widehat c^{\,\theta}_\pm(X,Y) &\coloneqq \pm\frac{1}{2}\left(\widehat l^{\,\theta}(X)+\widehat l^{\,\theta}(Y)-\widehat l^{\,\theta}(X\mp Y)\right).
\end{align}
Using $\mathrm{re}\,\widehat\varphi_X$, we eliminate imaginary errors, so the logarithms are real-valued functions. The case $\mathrm{re}\,\widehat\varphi_X < 0$ formally corresponds to $l^\theta(X) = \infty$; such estimmation results requires special care.
\end{dfn}

\begin{rem}
    The above estimators are consistent due to the sample mean being consistent. They are however only asymptotically unbiased due to $\ln\EE{X}\neq \EE{\ln X}$. Up to the first order the bias is  $\theta^{-2}\var(\mathrm{re}\,\widehat\varphi_X)/\varphi_X^2$.
\end{rem}

\begin{prp}[Asymptotic distribution of the empirical lcf and codifferences]\label{prp:ad}
In the limit of large sample sizes the lcf and codifference are asymptotically Gaussian
\begin{align}
    \sqrt{n}\Big(\widehat l^{\,\theta}(X)-l^\theta(X)\Big) &\xrightarrow{\ d\ } \mathcal N\left(0,\frac{4}{\theta^4}\frac{C_{X,X}(\theta)}{\varphi_X(\theta)^2}\right),\\
       \sqrt{n}\Big(\widehat s^{\,\theta}(X,Y)-s^\theta(X,Y)\Big) &\xrightarrow{\ d\ } \mathcal N\left(0,\frac{1}{4\theta^4}\left(\frac{C_{X+Y,X+Y}(\theta)}{\varphi_{X+Y}(\theta)^2}+ \frac{C_{X-Y,X-Y}(\theta)}{\varphi_{X-Y}(\theta)^2}-\frac{2C_{X+Y,X-Y}(\theta)}{\varphi_{X+Y}(\theta)\varphi_{X-Y}(\theta)}\right)\right)\\
        \sqrt{n}\Big(\widehat c_\pm^{\,\theta}(X,Y)-c_\pm^{\theta}(X,Y)\Big) &\xrightarrow{\ d\ } \mathcal N\left(0,\frac{1}{\theta^4}\left(\frac{C_{X,X}(\theta)}{\varphi_X(\theta)^2}+\frac{C_{Y,Y}(\theta)}{\varphi_Y(\theta)^2}+\frac{C_{X\mp Y,X\mp Y}(\theta)}{\varphi_{X\mp Y}(\theta)^2}+\right.\right. \nonumber\\
        &\qquad\qquad \left.\left. \frac{2C_{X,Y}(\theta)}{\varphi_{X}(\theta)\varphi_{Y}(\theta)}-\frac{2C_{X,X\mp Y}(\theta)}{\varphi_{X}(\theta)\varphi_{X\mp Y}(\theta)}-\frac{2C_{Y,X\mp Y}(\theta)}{\varphi_{Y}(\theta)\varphi_{X\mp Y}(\theta)}\right)\right) 
\end{align}
with $C_{A,B}$ defined in \eqref{eq:realCov}.
\end{prp}
\begin{proof}
    Variables $\widehat\varphi_X, \widehat\varphi_Y,\widehat\varphi_{X+Y},\widehat\varphi_{X-Y}$ are asymptotically jointly Gaussian (it is easy to see if one considers their linear combinations); the lcf and codifference are their smooth transformation of type $\ln(x)+\ln(y)$. By the application of the delta method we know they are asymptotically unbiased Gaussian with variances governed by the first order Taylor expansion terms which look like $(\widehat\varphi_X-\varphi_X)/\varphi_X$, etc. Their variance can be calculated as the sum of the pairwise covariances and conveniently expressed using \eqref{eq:realCov}.
\end{proof}
In Fig. \ref{fig:asymptDistr} we show numerically that reasonable similarity to the asymptotic Gaussian distribution is reached for samples with few hundred points. We expect this prediction to be robust given that $\e^{\I \theta X}$ is bounded by 1 no matter the distribution of $\theta X$. 

\begin{figure}\centering
	\includegraphics[width=0.9\textwidth]{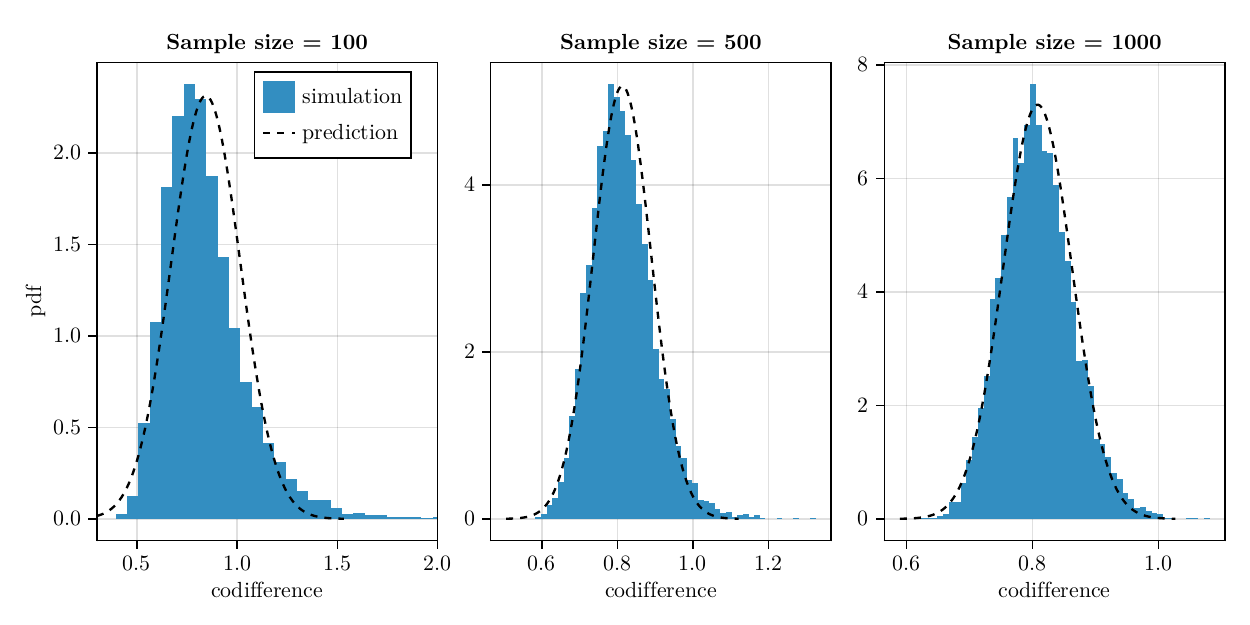}
	\caption{Monte Carlo simulations of symmetric codifference calculated from spherically symmetric 2D Laplace distribution (blue bars) and the asymptotic prediction Prp. \ref{prp:ad}. (black dashed line).}\label{fig:asymptDistr}
\end{figure}

\begin{rem}[Confidence intervals and tests]
First important application of the above result is that we can construct asymptotic confidence intervals and make tests. We substitute $\widehat \varphi$ for $\varphi$ in Prop. \ref{prp:ad} to estimate the required asymptotic variance $\widehat \sigma^2$ and then one can use standard formula $\pm z_{\alpha/2}\widehat \sigma/\sqrt{n}$. Similarly for tests, positive, negative, or zero codifference can be verified. Such test should be performed for a single fixed $\theta$ as codifference for varying $\theta$ is highly dependent.

\end{rem}

\begin{rem}[The empirical lcf and codifference for large scale variables]\label{rem:scale}
The formulas for asymptotic variance contain terms like $1/\varphi_X(\theta)$, which means that the estimation becomes unreliable for small $\varphi_X(\theta)$ which happens for distributions with a large scale or for large $\theta$. This is perhaps the most significant problem with using the codifference in practice.

Errors of estimating covariance are proportional to covariance, but for the lcf and codifference they increase quicker, for Gaussian variable like $\exp(\sigma^2\theta^2/2)$. This can be regulated with $\theta$, but for example plotting $\widehat l^{\,\theta}_X(t)$ for some spreading diffusion like L\'evy flight requires very huge sample sizes if we are interested in large times.

In the same limit of large scale of the estimated distribution the Gaussian approximation fails, as it requires the argument of logarithm to be close to 1. Values close to 0 will cause non-linear divergence and the negative ones numerical errors, which need to be carefully dealt with.
\end{rem}

As a final topic about estimation, let us note that for stationary processes there is an option to estimate the lcf and codifference using time averages. Given one stationary trajectory measured at equally spaced points $(X(t))_{t=1}^T$ the time-averaged lcf is just log of time averaged characteristic function
\begin{equation}
    \overline{l^\theta}(X) \coloneqq -\frac{2}{\theta^2}\ln\left(\frac{1}{T} \sum_{t=1}^T \e^{\I\theta X(t)}\right)
\end{equation}
similarly for the codifference, which has definition analogical to the time-averaged covariance
\begin{equation}
    \overline{s}^{\,\theta}_X(t)\coloneqq -\frac{1}{2\theta^2}\left(\ln\left(\sum_{\tau=1}^{T-t}\e^{\I(X(t+\tau)+X(\tau))}\right) - \ln\left(\sum_{\tau=1}^{T-t}\e^{\I(X(t+\tau)-X(\tau))}\right)  \right),
\end{equation}
analogically for $\overline{c}^{\,\theta}_\pm$.  Dividing by $T-t$ is not necessary, as it cancels out by subtraction of logarithms. If the process is ergodic, these quantities converge to their ensemble averaged equivalents. For non-ergodic processes, such as mixture/randomly parametrised processes this is still interesting to analyse as it returns valuable information about single trajectory properties and the nature of non-ergodicity.

Even if the process under consideration is not stationary, often its increments are and this methodology still can be used. Its limitations are very similar to standard covariance analysis, though it should be expected that larger amout of data may be required in some cases, as argued in Remark \ref{rem:scale}.

\section{Conclusions}

Our goal here was to facilitate the use as tools the lcf to study dispersion and the codifference to study dependence, without limiting the discussion to stable variables for which these concepts originated. The strongest practical limitation of the codifference is that its estimation in some cases requires large data samples; this is visible in the asymptotic distribution of its estimator which can become singular for small samples. Nevertheless, when sufficiently large samples are available its estimation is easy. We identify its two main advantages: it can detect forms of dependence which covariance cannot (e.g.  related to joint amplitude), and in contrast to the covariance it measures the properties of distribution's bulk more than its tails. We specify the meaning of this statement by showing the behaviour of codifference for a wide class of mixture distributions. These results open up a possibility for broader codifference based data analysis.

\appendix
\appendixpage

\section{Codifference on linear spaces}\label{app:linSpace}

\begin{dfn}[Covariance operator]
Let $\mathbb V$ be a vector space and $V$ a zero-mean $\mathbb V$-valued random variable. Let $a,b\in \mathbb V^*$ be vectors from the dual space, that is a linear, $\mathbb R$ or $\mathbb C$-valued functional on $\mathbb V$. Denoting by $\langle\bd\cdot,\bd\cdot\rangle$ a functional acting on a vector,  the covariance operator $R\colon \mathbb V^*\times \mathbb V^*\to \mathbb R$ is defined by the formula
\begin{equation}
R_V(a,b)\coloneqq r(\langle a, V\rangle, \langle V, b\rangle) = \EE{\langle a, V\rangle \langle V, b\rangle},
\end{equation}
which is finite if and only if $\EE{|\langle a, V\rangle|^2}<\infty$ for all $a\in\mathbb V^*$.
\end{dfn}
Usually $\mathbb V$ is a Hilbert space and $\langle\bd\cdot,\bd\cdot\rangle$ is just a scalar product, but we mention general case as it is sometimes useful to work on distributions. For example, the covariance function of a stochastic process $(X_t)_t$ can be interpreted as the covariance operator at Dirac deltas arguments, $r_X(s,t) = R_X(\delta(s),\delta(t))$.

In contrast to two argument covariance, the operator $R_V(a,b)$ is basis-independent, which is important for multiple applications. For example, if $V$ is a position or velocity vector the axes chosen are often accidental and the analysis should reflect that. This operator essentially quantifies how much of a rectangular symmetry along $a,b$ axes the distribution has, being a difference of squared amplitudes between the coordinates $\langle a+b,V\rangle$ and $\langle a-b,V\rangle$, note that for versors directions $a+b$ and $a-b$ are perpendicular. 

Naturally then, the codifference can be viewed geometrically as a difference in spread of $\langle a+b,V\rangle$ and $\langle a-b,V\rangle$, measured by the lcf. The obtained operator is non-linear, but still has a reasonable interpretation.

\begin{dfn}[Codifference operator]
With the same assumptions as above, for $V$ with positive definite distribution we define the codifference operator $T\colon \mathbb V^*\times \mathbb V^*\to \mathbb R$ as
\begin{equation}
T_V(a,b)\coloneqq s(\langle a,V\rangle, \langle V,b\rangle) =\frac{1}{2}  \ln\frac{\EE{\e^{\I\langle a-b,V\rangle}}}{\EE{\e^{\I\langle a+b,V\rangle}}}.
\end{equation}
\end{dfn}
As before, one can also calculate $T_V$ given $V$ with distribution which is only symmetric ($V\deq -V$) and not positive definite, but in this case $T_V(a,b)$ is finite only for sufficiently small $a$ and $b$. One could also consider operator given by $c_\pm(\langle a,V\rangle, \langle V,b\rangle)$  but due to asymmetry it would be substantially harder to interpret, so we skip it.

In analogy to the covariance, codifference of a stochastic process can be reinterpreted as $T_X(\delta(s),\delta(t))$. This is convenient as for models based on  L\'evy noises the functional characteristic function $\varphi_{\dd L}(f)\coloneqq \EE{\exp(\I \int\dd L\ f)}$ is useful for studying the distributional properties.

Most of the elementary properties of symmetric codifference pass to its geometrical extension.
\begin{prp}\label{prp:geomProp} The codifference operator
\begin{itemize}
\item[a)]Is symmetric with respect to $a,b$, $T_V(a,b)=T_V(b,a)$.
\item[b)] Is antisymmetric with respect to negation, $T_{V}(a,-b) = - T_V(a,b)$.
\item[c)] Is positive on diagonal, $T_V(a,a)>0, a\neq 0$. (Extension of positive definite property.)
\item[d)] Treats $V$ and $a,b$ as dual in the sense that for an operator $A\colon \mathbb V^*\to \mathbb V^*$
\begin{equation}
T_V(Aa,Ab) = T_{A^*V}(a,b),
\end{equation}
where $A^*$ is the dual of $A$.
\item[e)] Is additive for sums of independent variables. If $V=W+U$ for $W\indep U$
\begin{equation}
T_V(a,b) = T_W(a,b)+T_U(a,b).
\end{equation}
\end{itemize}
\end{prp}

\begin{prp}\label{prp:codCov} Relations between the covariance and codifference operators.
\begin{itemize}
\item[a)] For zero-mean Gaussian vectors 
\begin{equation}
T_V(a,b) = R_V(a,b).
\end{equation}

\item[b)] If $\EE{\langle a, V\rangle^2},\EE{\langle b, V\rangle^2}<\infty$, then
\begin{equation}
T_V(\epsilon a,\eta b)\sim R_V(\epsilon a,\eta b),\quad (\epsilon,\eta) \to 0
\end{equation}
\end{itemize}
\end{prp}
\begin{proof}
Point a) stems from $s(\langle a, V\rangle, \langle b, V\rangle) = r(\langle a, V\rangle, \langle b, V\rangle) $.
For b) we note that the limit 
\begin{equation}
T_V(\epsilon a,\eta b) =  \frac{1}{2}\ln\frac{\EE{\e^{\I\langle\epsilon a-\eta b,V\rangle}}}{\EE{\e^{\I \langle\epsilon a+\eta b,V\rangle}}}
\end{equation}
has type $\ln(1/1)$. We can use the fact that $\ln(1+x)\sim x,x\to 0$ and write
\begin{equation}
T_V(\epsilon a,\eta b) =  \frac{1}{2}\ln\left(1 + \frac{\EE{\e^{\I\langle\epsilon a-\eta b,V\rangle}- \e^{\I \langle\epsilon a+\eta b,V\rangle}}}{\EE{\e^{\I \langle\epsilon a+\eta b,V\rangle}}}\right).
\end{equation}
The denominator here converges to 1. Because $\I\EE{\langle \epsilon a\pm \eta b,V\rangle }=0$ we can add and subtract these terms at will in the numerator. The function $\exp(x)-x$ is bounded by $x^2$ near $x=0$ and $\exp(x)-x\sim x^2/2$. The functions under the expected value are therefore bounded by $\langle \epsilon a\pm \eta b,V\rangle^2$ which is integrable due to the assumptions. Using the dominated convergence theorem we may thus write
\begin{equation}
2T_V(\epsilon a,\eta b)\sim \I^2\EE{\langle \epsilon a - \eta b,V\rangle^2}-\I^2\EE{\langle \epsilon a\pm \eta b,V\rangle^2} = 2\EE{\langle \epsilon a, V\rangle\langle\eta b,V\rangle}= 2R_V(\epsilon a,\eta b).
\end{equation}
\end{proof}
Point b) is also true for sequences $(a_n,b_n)\to 0$ such that $\EE{\langle a_n,V\rangle^2},\EE{\langle b_n, V\rangle^2}$ are bounded. If this is not the case one can easily provide counterexamples. It is enough to have $\EE{\langle c,V\rangle^2}=\infty$ and approach the line $\epsilon c$ using $a_n$ and $b_n$ to obtain diverging results.

\begin{rem}[Domain of the codifference operator]
The codifference is well-defined for all positive-definite distributions, but in practical applications it should be applied for distributions with characteristic function unimodal along rays, otherwise it is hard to provide a reasonable interpretation of this quantity. Being unimodal along rays is a weak variant of unimodalidy: a multivariate function $\varphi$ is unimodal along rays when it is 1D unimodal along all possible straight lines originating at some $x_0$ (centre of the mode), that is all functions $s\mapsto \varphi(x_0+sv)$ are unimodal, i.e. non-decaying for $s<0$ and non-increasing for $s>0$. For an interested reader we provide few basic facts about this class in the Appendix \ref{app:unimodal}.  All the examples discussed have this property.
\end{rem}

\begin{ex}\label{ex:iid} Let us consider vector $V$ whose entries are made of linear combinations of some independent variables $X_j$ with univariate characteristic functions $\varphi_{X_j}$. Their number can be finite or infinite if the corresponding series converges in distribution. Then we can express $\langle a,V\rangle = \sum_j a_j X_j$ and $\langle a,V\rangle = \sum_j b_j X_j$. The codifference can be expressed by the formula
\begin{equation}
T_V(a,b) = \frac{1}{2}\sum_j \ln \frac{\varphi_{X_j}(a_j-b_j)}{\varphi_{X_j}(a_j+b_j)}.
\end{equation}
Thus, it measures difference between $|a_j-b_j|$ and $|a_j+b_j|$ weighted non-linearly. Borderline cases are clearly $a_j=b_j$ for positive, $a_j=-b_j$ for negative dependence. Zero codifference is observed when $a_j$ and $b_j$ have distinct support, $a_jb_j=0$ (variables $\langle a,V\rangle$ and $\langle b,V\rangle$ are then independent), but also if we balance out influences coming from different $X_j$. For example, if we can decompose $\langle a,V\rangle=X+Y, \langle b,V\rangle = X-Y$ for some iid $X$ and $Y$ then the variables $\langle a,V\rangle$ and $\langle b,V\rangle$ in general can be dependent, but the codifference will be zero. The latter property is unavoidable if we are interested in dependence measures which are antisymmetric with respect to the negation.
\end{ex}

\begin{ex}\label{ex:ell} Let random vector $V$ have elliptical distribution, that is, its characteristic function $\varphi$ can be expressed as
\begin{equation}
\varphi_V(\theta) = f( \langle\theta, \Sigma \theta\rangle),
\end{equation}
where $\Sigma$ is a non-negative definite matrix called the shape matrix; for the Gaussian case $\Sigma$ is, up to a constant, the covariance matrix. The probability density is then a function of $\langle x,\Sigma^{-1}x\rangle$. Important examples of elliptic distributions are Gaussian mixtures of the form $\sqrt{S}V$ like in Section \ref{s:gaussMix}.

Positive definite matrices are self-adjoint which shows that the codifference has the form
\begin{equation}
T_V(a,b) = \frac{1}{2}\ln \frac{f(\langle a,\Sigma a\rangle-2\langle a,\Sigma b\rangle+\langle b,\Sigma b\rangle)}{f(\langle a,\Sigma a\rangle+2\langle a,\Sigma b\rangle+\langle b,\Sigma b\rangle)}.
\end{equation}
For rotationally invariant (spherical) distributions $\Sigma = I$ and $T_V$ only measures difference in length between $a-b$ and $a+b$. In particular it is zero for any two perpendicular directions, reflecting symmetry of the system. Generalising this insight, it is also zero for any $a,b$ chosen from the set of principal axes of the ellipsis given by $\Sigma$ as in such a basis $\Sigma$ becomes diagonal and $\langle a,\Sigma b\rangle = 0$.

This is if and only if when $f$ is decreasing, that is $\varphi$ is unimodal (see Appendix \ref{app:unimodal}). In this case the
sign of $T_V$ agrees with the sign of $\langle a,\Sigma b\rangle$ and is a monotonic function of it.

Another way of looking at this picture is that we can decompose $V$ using the eigensystem $e_j,\lambda_j$ of $\Sigma$ with $e_j$ being the principal axes of the corresponding ellipsoid. In this basis the codifference operator becomes diagonal, $T_V(e_i,e_j) = 0$ for $i\neq j$, $T_V(e_i,e_i)=-\frac{1}{2}\ln f(4\lambda_i \Vert e_i\Vert^2)$. This is similar to Example \ref{ex:iid}, but this time $\langle e_i,V\rangle$ and $\langle e_j,V\rangle$ are in general non-linearly dependent.

\end{ex}

\section{Unimodal functions}\label{app:unimodal}
There is no single notion of unimodality for multivariate functions; the correct one depends on the applications that one have in mind. One of the strongest one is concavity. However, concave functions on infinite domains cannot have bounded minimum, so this notion cannot be applied globally to characteristic function. A weaker definition is that of log-concavity. A positive-valued function $f$ is log concave, if $\ln f$ is concave. This property is a requirement of many methods and algorithms related, e.g. to maximum likelihood estimation. An even weaker definition is that of quasiconcavity. Function $f$ is quasiconcave if all superlevel sets $\{x\colon f(x)\ge \alpha \}$ are convex. Directly from the definition of concave set we can see that it is equivalent to requiring that all functions $s\mapsto f(x+sv)$ are unimodal. 

 The property which is needed for proper behaviour of the codifference is even weaker: for a characteristic function we only need $s\mapsto \varphi(x_0+sv)$ to be unimodal for one given $x_0$, which clearly must be 0. In other words we require that all superlevel sets of $\varphi$ are star-shaped. This is unimodality along rays. It is not to be confused with ``$\varphi$ being star-shaped/star-concave'' which is another notion of weak unimodality  \cite{mathOpt}, weaker than quasiconcavity, but stronger than our requirement.
 
The most important point to take from the above discussion is that one can use any of the known criteria for the listed stronger variants of unimodality for the codifference. Many commonly used distributions are unimodal along rays, the most prominent example being symmetric stable distributions with the characteristic function $\varphi_X(sv) = \exp(-c_v|s|\alpha)$. Examples of distributions which do not have this property include any discrete distribution (because their characteristic functions are periodic). As an example of positive-definite continuous distribution without this property one can consider pdf $p(x) \propto \exp(-x^2)(1-|x|)\boldsymbol{1}_{|x|<1}$, its Fourier transform has decaying positive oscillations.

It is also worth to note that it is easy to provide examples of distributions which have unimodal among rays characteristic function (uarcf) constructed using more basic ones.
\begin{prp}[Classes of distributions with uarcf]\label{prp:unimodBase}\

\begin{itemize}
\item[a)] The sum of independent   variables with uarcfs have uarcf. In particular, any linear transformation of i.i.d. vector with uarcf  have uarcf.
\item[b)] The mixture of variables with uarcfs have uarcf.
\item[c)] The product of a variable with uarcf  and any other independent scalar variable have uarcf.
\item[d)] Elliptic distributions have uarcf if and only if any of the marginal distributions have.
\end{itemize}
\end{prp}
\begin{proof} Take $Y = X_1 +X_2$, then the characteristic functions can be written as a product $\varphi_Y(sv) = \varphi_{X_1}(sv)\varphi_{X_2}(sv)$ and is clearly monotonically increasing for $s <0 $ and decreasing for $s > 0 $. Similarly, let $(X_c)_c$ be a family of variables with uarcfs indexed by $c$ and $C$ be a random variable onto the space of $c$, $C\indep (X_c)_c$. Writing $\varphi_{X_C}(s v) = \EE{\varphi_{X_c}(s v)}$ shows that $\varphi_{X_C}$ inherits the monotonicity of $\varphi_{X_c}$. For the last point note that $X_c\coloneqq c X$ is a family of variables with characteristic functions $\varphi_{X_c} =\varphi_{X}(s c v)$, so they are uarcf if $X$ has uarcf. At last, elliptic distributions are linear transformations of rotationally invariant distributions. For the latter ones $\varphi_{X}(sv)$ does not depend on $v$, the result follows. 
\end{proof}

\bibliographystyle{ieeetr}
\bibliography{bibS}

@article{varGamma,
    author = {Adrian Fischer and Robert E. Gaunt and Andrey Sarantsev},
    title = {{The Variance-Gamma Distribution: A Review}},
    volume = {40},
    fjournal = {Statistical Science},
    journal = {Stat. Sci.},
    number = {2},
    publisher = {Institute of Mathematical Statistics},
    pages = {235 -- 258},
    year = {2025},
    doi = {10.1214/24-STS929},
    URL = {https://doi.org/10.1214/24-STS929}
}

@article{rosinski97,
  title = "The Equivalence of Ergodicity and Weak Mixing for Infinitely Divisible Processes",
  volume = {10},
  ISSN = {0894-9840},
  url = {http://dx.doi.org/10.1023/A:1022690230759},
  DOI = {10.1023/a:1022690230759},
  number = {1},
  journal = {J. Theor. Probab.},
  publisher = {Springer},
  author = {Rosiński,  Jan and Żak,  Tomasz},
  year = {1997},
  pages = {73–86}
}

@article{kruczek,
    title = {The modified Yule-Walker method for $\alpha$-stable time series models},
    journal = {Physica A},
    volume = {469},
    pages = {588-603},
    year = {2017},
    issn = {0378-4371},
    doi = {https://doi.org/10.1016/j.physa.2016.11.037},
    url = {https://www.sciencedirect.com/science/article/pii/S0378437116308421},
    author = {Piotr Kruczek and Agnieszka Wyłomańska and Marek Teuerle and Janusz Gajda},
}

@article{feuerverger,
  title = {The Empirical Characteristic Function and Its Applications},
  volume = {5},
  ISSN = {0090-5364},
  url = {http://dx.doi.org/10.1214/aos/1176343742},
  DOI = {10.1214/aos/1176343742},
  number = {1},
  fjournal = {The Annals of Statistics},
  journal = {Ann. Stat.},
  publisher = {Institute of Mathematical Statistics},
  author = {Feuerverger,  Andrey and Mureika,  Roman A.},
  year = {1977},
}

@article{slezak2019,
  title = {Codifference can detect ergodicity breaking and non-Gaussianity},
  volume = {21},
  ISSN = {1367-2630},
  url = {http://dx.doi.org/10.1088/1367-2630/ab13f3},
  DOI = {10.1088/1367-2630/ab13f3},
  number = {5},
  fjournal = {New Journal of Physics},
    journal = {New J. Phys.},
  publisher = {IOP Publishing},
  author = {\'Sl\k{e}zak,  Jakub and Metzler,  Ralf and Magdziarz,  Marcin},
  year = {2019},
  pages = {053008}
}

@article{kolmogorov,
    author = {Kolmogorov, Andrey},
    title = {Sur la notion de la moyenne},
    journal = {Atti. Accad. Naz. Lincei },
    volume = {12},
    pages = {388-391},
    year = {1930},
}

@article{shao93,
    author = {Shao, M. and Nikias, C. L.},
    title = {Signal processing with fractional lower order moments: Stable processes and their applications},
    fjournal = {Proceedings of the IEEE},
    journal = {Proceed. IEEE},
    volume = {81},
    pages = {986--1010},
    year = {1993},
    doi={10.1109/5.231338},
}

@article{nowicka97-1,
    author = {Nowicka, J. and Weron, A.},
    title = {Measures of dependence for {ARMA} models with stable innovations},
    fjournal = {Annales UMCS, Section A (Mathematica)},
    journal = {Ann. UMCS Sec. A},
    volume = {51},
    pages = {133--144},
    year = {1997},
}

@article{nowicka97-2,
    author = {Nowicka, J.},
    title = {Asymptotic behavior of the covariation and the codifference for {ARMA} models with stable innovations},
    journal = {Comm. Stat. Stoch. Mod.},
    volume = {13},
    pages = {673--686},
    year = {1997},
    doi = "10.1080/15326349708807446",
}

@article{magdziarz11,
    title = "Ergodic properties of anomalous diffusion processes",
    author = "Magdziarz, Marcin  and Weron, Aleksander ",
    fjournal = "Annals of Physics",
    journal = "Ann. Phys.",
    volume = "326",
    number = "9",
    pages = "2431 - 2443",
    year = "2011",
    issn = "0003-4916",
    doi = "https://doi.org/10.1016/j.aop.2011.04.015",
}

@article{loch16,
    title = {Ergodicity testing using an analytical formula for a dynamical functional of alpha-stable autoregressive fractionally integrated moving average processes},
    author = {Loch, Hanna and Janczura, Joanna and Weron, Aleksander},
    journal = {Phys. Rev. E},
    volume = {93},
    issue = {4},
    pages = {043317},
    numpages = {10},
    year = {2016},
    doi = {10.1103/PhysRevE.93.043317},
}

@article{loch18,
    author = {Loch-Olszewska,Hanna  and Szwabi{\'n}ski, Janusz },
    title = {Detection of $\epsilon$-ergodicity breaking in experimental data-A study of the dynamical functional sensibility},
    fjournal = {The Journal of Chemical Physics},
    journal = {J. Chem. Phys.},
    volume = {148},
    number = {20},
    pages = {204105},
    year = {2018},
    doi = {10.1063/1.5025941},
}

@article{loch19,
    title = "Properties and distribution of the dynamical functional for the fractional {Gaussian} noise",
    author = "Loch-Olszewska, Hanna ",
    fjournal = "Applied Mathematics and Computation",
    journal = "Appl. Math. Comput.",
    volume = "356",
    pages = "252 - 271",
    year = "2019",
    doi = "10.1016/j.amc.2019.03.038",
}

@article{chechkinCod,
    author = {Wy{\l}oma{\'n}ska, A. and Chechkin,  A. and Gajda. J. and Sokolov, I. M.},
    title = {Codifference as a practical tool to measure interdependence},
    fjournal = {Physica A: Statistical Mechanics and Its Applications},
    journal = "Physica A",
    volume = {421},
    pages = {412},
    year = {2015}
}

@incollection{podgorski91,
    author="Podg{\'o}rski, Krzysztof
    and Weron, Aleksander",
    title="Characterizations of ergodic stationary stable processes via the dynamical functional",
    booktitle="Stable Processes and Related Topics: A Selection of Papers from the Mathematical Sciences Institute Workshop, January 9--13, 1990",
    editor="Cambanis, Stamatis
    and Samorodnitsky, Gennady
    and Taqqu, Murad S.",
    year="1991",
    publisher="Birkh{\"a}user Boston",
    _address="Boston, MA",
    pages="317--328",
    doi="10.1007/978-1-4684-6778-9_16",
}

@incollection{weron84,
    author = {Weron, A.},
    title = {Stable processes and measures; a survey},
    year = 1984,
    booktitle = {Probability Theory on Vector Spaces III.},
    editor = {Szynal, D. and Weron, A.},
    publisher = "Springer",
    pages = "306--364",
    place = "Berlin",
}

@article{miller78,
    author = {Miller, G.},
    title = {Properties of certain symmetric stable distributions},
    fjournal = {Journal of Multivariate Analysis},
    journal = {J. Multivar. Anal.},
    volume = {8},
    pages = {346},
    year = {1978}
}

@article{gallagher01,
    author = {Gallagher, C. M.},
    title = {A method for fitting stable autoregressive models using the autocovariation function},
    fjournal = {Statistics \& Probability Letters},
    journal = {Stat. Prob. Lett.},
    volume = {53},
    pages = {381},
    year = {2001}
}

@article{zak15,
    author={{\.Z}ak, G. and Wy{\l}oma{\'n}ska, A. and Zimroz, R.},
    title={Application of alpha-stable distribution approach for local damage detection in rotating machines},
    fjournal={Journal of Vibroengineering},
    journal={J. Vibroeng.},
    year={2015},
    volume={17},
    number={6},
    pages={2987-3002},
}

@article{zak19,
    title = "Periodically impulsive behavior detection in noisy observation based on generalized fractional order dependency map",
    author={{\.Z}ak, G. and Wy{\l}oma{\'n}ska, A. and Zimroz, R.},
    fjournal = "Applied Acoustics",
    journal = "Appl. Acoust.",
    volume = "144",
    pages = "31 - 39",
    year = "2019",
    note = "Special Issue on ICTD-CMMNO'2016 ( ICTD : International Congress on Technical Diagnostics)",
    issn = "0003-682X",
    doi = "10.1016/j.apacoust.2017.05.003",
}

@article{crossCod,
    author = {Grzesiek, Aleksandra  and  Teuerle, Marek and Wy{\l}oma{\'n}ska, Agnieszka },
    title = {Cross-codifference for bidimensional {VAR(1)} time series with infinite variance},
    fjournal = {Communications in Statistics - Simulation and Computation},
    journal = {Commun. Stat. - Simul. Comput.},
    pages = {1--26},
    year  = {2019},
    publisher = {Taylor & Francis},
    doi = {10.1080/03610918.2019.1670840},
}

@article{cumulants,
    author = {McCullagh, Peter and Kolassa, John},
    title = {Cumulants},
    journal = {Scholarpedia},
    volume = {4},
    issue = 3,
    pages = {4699},
    year = {2009},
    doi = "10.4249/scholarpedia.4699",
}

@InProceedings{codi1,
    author="L{\'e}vy, Joshua B.
    and Taqqu, Murad S.",
    editor="Cambanis, Stamatis
    and Samorodnitsky, Gennady
    and Taqqu, Murad S.",
    title="A Characterization of the Asymptotic Behavior of Stationary Stable Processes",
    booktitle="Stable Processes and Related Topics: A Selection of Papers from the Mathematical Sciences Institute Workshop, January 9--13, 1990",
    year="1991",
    publisher="Birkh{\"a}user Boston",
    _address="Boston, MA",
    pages="181--198"
}

@InProceedings{codi2,
    author="Podg{\'o}rski, Krzysztof
    and Weron, Aleksander",
    editor="Cambanis, Stamatis
    and Samorodnitsky, Gennady
    and Taqqu, Murad S.",
    title="Characterizations of ergodic stationary stable processes via the dynamical functional",
    bookTitle="Stable Processes and Related Topics: A Selection of Papers from the Mathematical Sciences Institute Workshop, January 9--13, 1990",
    year="1991",
    pages="317--328",
    publisher="Birkh{\"a}user Boston",
    _address="Boston, MA"
}

@article{Grzesiek2020,
    author = {Grzesiek, Aleksandra and Sikora, Grzegorz and Teuerle, Marek and Wy{\l}oma{\'n}ska, Agnieszka},
    title = {Spatio-Temporal Dependence Measures for Bivariate {AR}(1) Models with $\alpha$-Stable Noise},
    fjournal = {Journal of Time Series Analysis},
    journal = {J. Time Ser. Anal.},
    volume = {41},
    number = {3},
    pages = {454-475},
    doi = {10.1111/jtsa.12517},
    year = {2020}
}

@book{copulaBook, 
	author = "Durante, Fabrizio and Sempi, Carlo", 
	title = "Principles of Copula Theory", 
	publisher = "Chapman \& Hall",
	address = "London",
	edition = "1st",
	year = 2015,
}

@book{lukacs, 
	author = " Lukacs, E.", 
	title = "Characteristic Functions", 
	publisher = "Charles Griffin \& Co.",
	_address = "London",
	edition = "2nd",
	year = 1970,
}

@book{inequalities,
	author= {Hardy, G. H. and Littlewood, J. E. and P\'olya, G.},
	edition = "2nd",
	year =   1952,
	title = "Inequalities",
	publisher =  " Cambridge University Press",
	_address= "Cambridge",
}

@book{mathOpt,
	author= {Pallaschke, D. and Rolewicz, S.},
	year =  1997,
	title = "Foundations of Mathematical Optimization",
	publisher =  "Springer Netherlands",
	doi = "10.1007/978-94-017-1588-1"
}

@book{taqqu, 
	author = "Samorodnitsky, G. and Taqqu, M. S.", 
	title = "Stable Non-{Gaussian} Random Processes", 
	publisher = "Chapman \& Hall",
	_address = "London",
	year = 1994,
}

@book{janicki, 
	author = "Janicki, Aleksander and Weron, Aleksander", 
	title = "Simulation and Chaotic Behavior of Alpha-stable Stochastic Processes", 
	publisher = "Marcel Dekker",
	_address = "New York",
	year = 1994,
}

@article{solarFlare,
	title = "From solar flare time series to fractional dynamics",
	journal = "Physica A",
	volume = "387",
	number = "5",
	pages = "1077 - 1087",
	year = "2008",
	doi = "10.1016/j.physa.2007.10.024",
	author = "Burnecki, K. and Klafter, J. and Magdziarz, M. and Weron, A.",
}

@article{solarFlare2,
  title = {{FARIMA} modeling of solar flare activity from empirical time series of soft x-ray solar emission},
  volume = {693},
  ISSN = {1538-4357},
  url = {http://dx.doi.org/10.1088/0004-637X/693/2/1877},
  DOI = {10.1088/0004-637x/693/2/1877},
  number = {2},
  fjournal = {The Astrophysical Journal},
  journal = {Astrophys. J.},
  publisher = {American Astronomical Society},
  author = {Stanislavsky,  A. A. and Burnecki,  K. and Magdziarz,  M. and Weron,  A. and Weron,  K.},
  year = {2009},
  month = mar,
  pages = {1877–1882}
}
\end{document}